\documentclass[11pt,twoside]{article}
\usepackage{fancyhdr}
\usepackage[colorlinks,citecolor=blue,urlcolor=blue,linkcolor=blue,bookmarks=false]{hyperref}
\usepackage{amsfonts,epsfig,graphicx}
\usepackage{afterpage}
\usepackage{comment}
\usepackage{amsmath,amssymb,amsthm} 
\usepackage{fullpage}
\usepackage[T1]{fontenc} 
\usepackage{epsf} 
\usepackage{graphics} 
\usepackage{amsfonts,amsmath}
\usepackage[sort]{natbib} 
\usepackage{psfrag,xspace}
\usepackage{color,etoolbox}
\usepackage{subcaption}

\def\inprob{\stackrel{p}{\rightarrow}}

\def\ind{\perp\!\!\!\perp}
\def\T{{ \mathrm{\scriptscriptstyle T} }}

\newcommand{\var}{\text{var}}
\newcommand{\cov}{\text{cov}}

\newcommand{\Pb}{\mathbb{P}}
\newcommand{\Pn}{\mathbb{P}_n}

\newcommand{\E}{\mathbb{E}}
\newcommand{\R}{\mathbb{R}}

\def\logit{\text{logit}}
\def\expit{\text{expit}}

\DeclareMathOperator*{\argmin}{arg\,min}

\DeclareSymbolFont{bbold}{U}{bbold}{m}{n}
\DeclareSymbolFontAlphabet{\mathbbold}{bbold}
\newcommand{\one}{\mathbbold{1}}
\newtheorem{theorem}{Theorem}
\newtheorem{lemma}{Lemma}
\newtheorem{corollary}{Corollary}
\newtheorem{proposition}{Proposition}

\newtheorem{algorithm}{Algorithm}
\theoremstyle{definition}
\newtheorem{definition}{Definition}

\theoremstyle{remark}

\newtheorem{remark}{Remark}
\newtheorem{condition}{Condition}

\begin{document}

\def\spacingset#1{\renewcommand{\baselinestretch}%
{#1}\small\normalsize} \spacingset{1}

\raggedbottom
\allowdisplaybreaks

%%%%%%%%%%%%%%%%%%%%%%%%%%%%%%%%%%%%%%%%%%%

  \title{\vspace*{-.4in} {Towards Optimal Doubly Robust Estimation \\ of Heterogeneous Causal Effects }}
  \author{\\ Edward H. Kennedy \\ \\
    Department of Statistics \& Data Science \\
    Carnegie Mellon University \\ \\ 
    \texttt{edward@stat.cmu.edu} \\
\date{}
    }

  \maketitle
  \thispagestyle{empty}

\begin{abstract}
Heterogeneous effect estimation is crucial in causal inference, with applications across medicine and social science. Many methods for estimating conditional average treatment effects (CATEs) have been proposed, but there are gaps in understanding if and when such methods are optimal. This is especially true when the CATE has nontrivial structure (e.g., smoothness or sparsity). Our work contributes in several ways. First, we study a two-stage doubly robust CATE estimator and give a generic error bound, which yields rates faster than much of the literature. We apply the bound to derive error rates in smooth nonparametric models, and give sufficient conditions for oracle efficiency. Along the way we give a general error bound for regression with estimated outcomes; this is the second main contribution. The third contribution is aimed at understanding the fundamental statistical limits of CATE estimation. To that end, we propose and study a local polynomial adaptation of double-residual regression. We show that this estimator can be oracle efficient under even weaker conditions, and we conjecture that they are minimal in a minimax sense. We go on to give error bounds in the non-trivial regime where oracle rates cannot be achieved. Some finite-sample properties are explored with simulations.
\end{abstract}

\noindent
{\it Keywords: conditional effects, influence function, minimax rate, nonparametric regression.} 

\section{Introduction}

Heterogeneous effect estimation plays a crucial role in causal inference, with applications across medicine and social science, e.g., improving understanding of variation, and informing policy or optimizing treatment decisions. The most common target parameter in this setup is the conditional average treatment effect (CATE) function, $\E(Y^1-Y^0 \mid X=x)$, which measures the expected difference in outcomes had those with covariates $X=x$ been treated versus not. The CATE is typically identified under standard causal assumptions (including no unmeasured confounding) as the difference between two regression functions, $\tau(x) \equiv \E(Y \mid X=x,A=1)-\E(Y \mid X=x,A=0)$. \\

 Important early methods for estimating the CATE often employed semiparametric models, for example partially linear models  assuming $\tau(x)$ to be constant \citep{robinson1988root, robins1992estimating}, or structural nested models in which $\tau(x)$ followed some known parametric form \citep{robins1994correcting, van2003unified, van2006statistical, vansteelandt2014structural}. These approaches reflect a commonly held conviction that the CATE  may be more structured and simple than the rest of the data-generating process (with the zero treatment effect case an obvious example). This is similar in spirit to the common belief that interaction terms are more often zero than ``main effects''. \\

Recent years have seen a move towards more flexible estimators of $\tau(x)$. The first formal nonparametric model where the CATE has its own complexity separate from regression functions seems to be Example 4 of \citet{robins2008higher}. \citet{van2006statistical} (Section 4.2) proposed an important model-free ``meta-algorithm'' for estimating the CATE, a variant of which we study in this paper. The last 5--10 years has seen even more emphasis on nonparametrics and incorporating machine learning \citep{foster2011subgroup, imai2013estimating, athey2016recursive, shalit2017estimating, semenova2017estimation, nie2017quasi, wager2018estimation, kunzel2019metalearners,  foster2019orthogonal, hahn2020bayesian}. The present work is in a similar vein, focusing on (i) providing more flexible CATE estimators with stronger theoretical guarantees, and (ii) pushing forward our understanding of optimality and the fundamental limits of CATE estimation. At the start of each of the Sections \ref{sec:genoracle}, \ref{sec:drlearner}, and \ref{sec:lprlearner}, we detail related work and describe how our results fit. \\

After describing the setup and presenting a simple motivating illustration in Section \ref{sec:setup}, we go on to present our three main contributions: (i) a model-free oracle inequality for regression with estimated pseudo-outcomes (given in Section \ref{sec:genoracle}); (ii) general conditions for  oracle efficiency of a doubly robust estimator we term the DR-Learner, with applications to specific regression methods under smoothness  conditions (in Section \ref{sec:drlearner});  and (iii) a more refined analysis of a specialized estimator we call the lp-R-Learner, showing that faster rates can be achieved in the non-oracle regime, and giving a partial answer towards understanding the fundamental limits of CATE estimation (in Section \ref{sec:lprlearner}). \\

\section{Setup \& Illustration} \label{sec:setup}

We assume access to an iid sample of observations of $Z_i=(X_i,A_i,Y_i)$, where $X \in \R^d$ are covariates, $A \in \{0,1\}$ is a binary treatment or exposure, and $Y \in \R$ an outcome of interest. The distribution of $Z$ is indexed by the covariate distribution and  nuisance functions:
\begin{align*}
\pi(x) &= \Pb(A =1 \mid X=x) \\
\mu_a(x) &= \E(Y \mid X=x, A=a) \\
\eta(x) &= \E(Y \mid X=x) 
\end{align*}
Our goal is estimation of the difference in regression functions under treatment versus control  $\tau(x) \equiv \mu_1(x) - \mu_0(x)$. 
Under standard causal assumptions of no unmeasured confounding, consistency, and positivity or overlap ($\epsilon \leq \pi \leq 1-\epsilon$ wp1, which we assume throughout),  the function $\tau(x)$ also equals  $\E(Y^1 - Y^0 \mid X=x)$, where $Y^a$ is the counterfactual outcome under $A=a$. We refer to $\mu_1(x)-\mu_0(x)$ as the CATE,  noting that our results hold regardless of whether the causal assumptions do. The average treatment effect (ATE) is given by $\E\{\tau(X)\}$. 

\subsection{Notation}

We use $\Pn(f)=\Pn\{f(Z)\} = \frac{1}{n} \sum_i f(Z_i)$ as shorthand for sample averages. When $x \in \R^d$ we let $\| x \|^2=\sum_j x_j^2$ denote the usual (squared) Euclidean norm, and for generic (possibly random) functions $f$ we let $\| f \|^2= \int f(z)^2 \ d\Pb(z)$ denote the (squared) $L_2(\Pb)$ norm. 
We use the notation $a \lesssim b$ to mean $a \leq Cb$ for some universal constant $C$, and $a \asymp b$ to mean $cb \leq a  \leq Cb$ so that $a \lesssim b$ and $b \lesssim a$. We let $a_n \sim b_n$ mean $a_n/b_n \rightarrow 1$ as $n \rightarrow \infty$. \\

At various points we refer to $s$-smooth functions, which we define as those contained in the H\"{o}lder class $\mathcal{H}(s)$,  a canonical function class in nonparametric regression, density estimation, and functional estimation.  Intuitively, it contains smooth functions that are close to their $\lfloor s \rfloor$-order Taylor approximations. More precisely, $\mathcal{H}(s)$ is the set of functions $f:\mathcal{X} \rightarrow \R$ that are $\lfloor s \rfloor$-times continuously differentiable with partial derivatives bounded, and for which
$$ | D^m f(x) - D^m f(x')| \lesssim \| x - x' \|^{s-\lfloor s \rfloor}  $$
for all $x,x'$ and $m=(m_1,...,m_d)$ such that $\sum_j m_j=\lfloor s \rfloor$, where $D^m=\frac{\partial^{\lfloor s\rfloor}}{\partial_{x_1}^{m_1}...\partial_{x_d}^{m_d}}$ is the multivariate partial derivative operator. \\

\subsection{Simple Motivating Illustration} \label{sec:motiv}

Consider a simple data-generating process where the covariates $X$ are uniform on $[-1,1]$, 
\begin{align*}
\pi(x) &= 0.5 + 0.4 \ \text{sign}(x) 
\end{align*}
and $\mu_1(x)=\mu_0(x)$ are equal to the piecewise polynomial function defined on page 10 of \citet{gyorfi2002distribution}, given by
$$ \mu_a(x) = \begin{cases}
(x+2)^2/2 & \text{ if } -1 \leq x \leq -0.5, \\
x/2 + 0.875 & \text{ if } -0.5 \leq x < 0, \\
-5(x-0.2)^2  +1.075 & \text{ if } 0 < x \leq 0.5, \\
x + 0.125 &  \text{ if } 0.5 \leq x < 1,
\end{cases} , $$ 
which is illustrated in Figure \ref{fig:simex1}. Figure \ref{fig:simex1} also shows $n=1000$ simulated data points from this data-generating process, approximately half of which are treated (shown on the left panel) and the other half untreated (shown on the right). Also shown are estimates of the corresponding $\mu_1$ and $\mu_0$ functions, using default tuning parameters with the \verb|smoothing.spline| function in base R. \\

An interesting but likely common phenomenon occurs in this simple example. The individual regression functions are non-smooth, and difficult to estimate well on their own; this is especially true in the region where there are fewer treated indviduals. Thus the estimate $\widehat\mu_1$ tends to oversmooth on the left, where there are more untreated individuals; in contrast, the estimate $\widehat\mu_0$ tends to \emph{undersmooth} on the right, where there are more treated individuals. This means a naive plug-in estimator of the CATE that simply takes the difference $\widehat\mu_1-\widehat\mu_0$ will be a poor and overly complex estimator of the true difference, which is not only a constant but zero.  \\

\begin{figure}[h!]
\begin{center}
{\includegraphics[width=\textwidth]{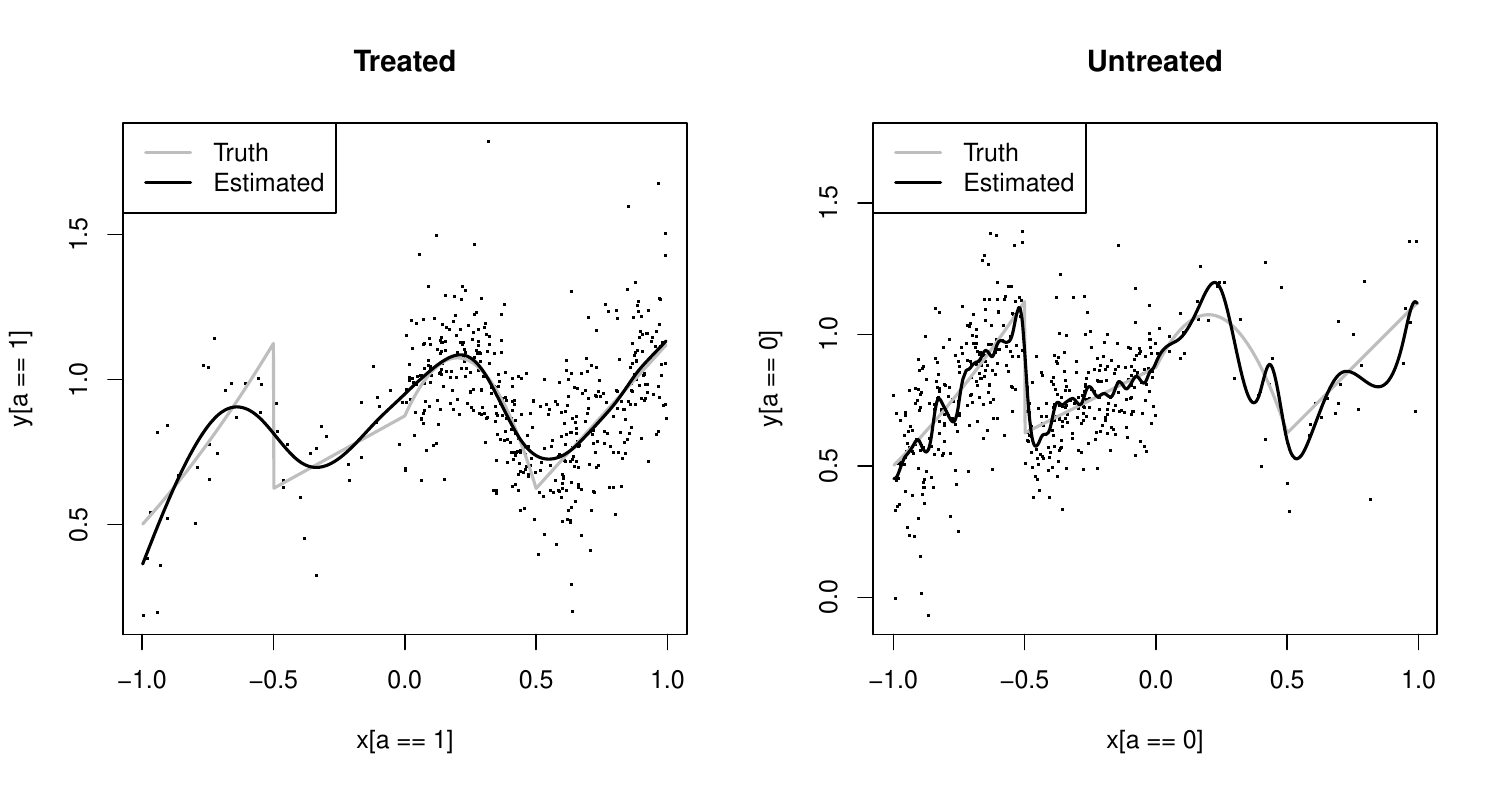}}
\caption{Plot of simulated data where the regression functions $\mu_1$ and $\mu_0$ individually are complex and difficult to estimate, but their difference is simply constant and equal to zero. Thus a naive plug-in estimator of the CATE will be overly complex, yielding large errors. } \label{fig:simex1}
\end{center}
\end{figure}

In contrast, suppose for simplicity that the propensity scores $\pi$ were known. Then a regression of the inverse-probability-weighted (IPW) pseudo-outcome $\frac{(A-\pi)Y}{\pi(1-\pi)}$ would, up to constants, behave just as an oracle estimator that had access to the actual counterfactual difference $Y^1-Y^0$, since the conditional mean of this pseudo-outcome is exactly $\tau(x)$. Figure \ref{fig:simex2} shows results from this procedure, as well as two other more efficient and doubly robust versions described in subsequent sections, again all using default tuning parameter choices from \verb|smoothing.spline|.  For these simulated data, the doubly robust estimators are much more efficient than the IPW estimator, and do a much better job of adapting to the correct  underlying simplicity of the true $\tau$. \\

\begin{figure}[h!]
\begin{center}
{\includegraphics[width=.65\textwidth]{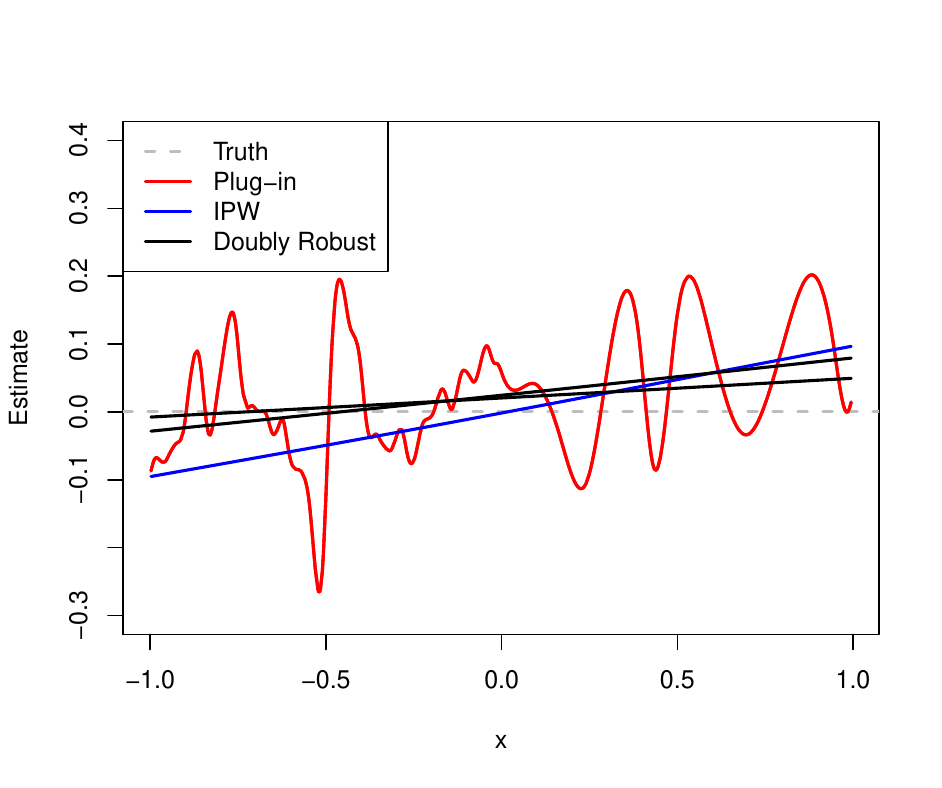}}
\caption{Estimated CATE curves in a simple simulated example. The plug-in method inherits the large errors from estimating the individual regression functions, which are complex and non-smooth. The IPW and doubly robust methods adapt to the smoothness of the CATE, which is constant in this example, with the doubly robust methods more efficient.} \label{fig:simex2}
\end{center}
\end{figure}

The results shown in Figure \ref{fig:simex2} are typical for this data-generating process: across 500 simulations, the IPW and doubly robust estimators gave smaller integrated squared bias across $X$ by a factor of 10--100, respectively, and the doubly robust estimators improved on the integrated variance of the IPW estimator by nearly a factor of 20. \\

\begin{remark} 
If the CATE has no additional structure (e.g., smoothness, sparsity) beyond that of the individual regression functions $\mu_a$, then a simple plug-in estimator $\widehat\tau(x)=\widehat\mu_1(x)-\widehat\mu_0(x)$ would be sufficient. However, in general the complexities of $\mu_a$ and $\tau$ may differ substantially, and in practice we often expect the CATE to be much more structured, as described in the Introduction.   \\ 
\end{remark}

In the following sections we study the error of procedures like those illustrated above, giving model-free guarantees for practical use, as well as a more theoretical study of the fundamental limits of CATE estimation in a large nonparametric model.   \\

\section{General Pseudo-Outcome Regression} \label{sec:genoracle}

In this section we lay out a framework for characterizing the error of general  two-stage regression estimators that regress estimated ``pseudo-outcomes'' on a covariate vector. We define a new notion of stability (which we prove holds for generic linear smoothers, for example), and then show how this stability can be combined with a small-bias condition to yield oracle efficiency. 
In addition to laying a foundation for the analysis of a doubly robust CATE estimator in the next section, the result should also be of independent interest in other problems  involving regression with estimated or imputed outcomes \citep{fan1994censored, ai2003efficient, rubin2005general, rubin2006doubly, kennedy2017nonparametric, bibaut2017data}. \\

With a few exceptions, previous work on regression with estimated outcomes has not appeared to exploit sample splitting, and has largely focused on particular pseudo-outcomes, and particular regression estimators (in both stages); in contrast we lean on sample splitting in order to obtain a more general result that is agnostic about the methods used, beyond a basic stability condition. Prominent examples of previous work include \citet{ai2003efficient}, \citet{rubin2005general}, and \citet{foster2019orthogonal}, all of which gave results for pseudo-outcomes of a general form. \citet{ai2003efficient} and \citet{rubin2005general} did not use sample splitting and so limited their attention to particular estimators. \citet{ai2003efficient} used sieves, and focused more on a finite-dimensional component appearing in the pseudo-outcome. \citet{rubin2005general} considered least squares, penalized methods, and linear smoothers in the second stage, while being agnostic about the first-stage regression. However, their error bounds do not allow one to exploit double robustness in the pseudo-outcome, which is a crucial advantage in practice and which our result does allow.  \\

Our results are closest in spirit to \citet{foster2019orthogonal}, who study generic empirical risk minimization when the loss involves complex nuisance functions. However there are some important distinctions to be made. Most importantly, \citet{foster2019orthogonal} assume the loss satisfies a Neyman orthogonality property in order to obtain squared error rates; in contrast, our bound does not rely on this structure, but can exploit it when it holds, as   shown in the following section. Crucially, our bound will also be seen to yield doubly robust errors, whereas the orthogonality-based results of \citet{foster2019orthogonal} yield errors that are second-order but not doubly robust. Double robustness is essential when different nuisance components are estimated with different errors. Lastly,  \citet{foster2019orthogonal} pursue global error bounds, while we consider pointwise error, which can be beneficial for developing inferential tools. These and some other differences are discussed further in the next section. \\

In what follows we first define a new notion of estimator stability, which can be viewed as a form of stochastic equicontinuity for  nonparametric regression (as opposed to averaging).  \\

\begin{definition}[Stability] \label{def:stability}
Suppose $D^n=(Z_{01},...,Z_{0n})$ and $Z^n=(Z_1,...,Z_n)$ are independent training and test samples, respectively, with covariate $X_i \subset Z_i$. 
Let:
\begin{enumerate}
%\item $D^n=(Z_{01},...,Z_{0n})$ and $Z^n=(Z_1,...,Z_n)$ denote independent training and test samples, respectively, with covariate denoted $X_i \subset Z_i$, 
\item  $\widehat{f}(z)=\widehat{f}(z;D^n)$ be an estimate of a function $f(z)$ using the training data $D^n$,
\item $\widehat{b}(x) = \widehat{b}(x;D^n)  \equiv \E\{ \widehat{f}(Z) - f(Z) \mid D^n, X=x\}$ the conditional bias of the estimator $\widehat{f}$, 
\item $\widehat\E_n(Y \mid X=x)$ denote a generic regression estimator that regresses outcomes $(Y_1,...,Y_n)$ on covariates $(X_1,...,X_n)$ in the test sample $Z^n$. 
%\item $R_n^*(x)^2 = {\E\left(\left[ \widehat\E_n\{ f(Z) \mid X=x\} - \E\{f(Z) \mid X=x\} \right]^2  \right)}$ the  mean squared error of the oracle estimator $\widehat\E_n\{f(Z) \mid X=x\}$. 
\end{enumerate}
Then the regression estimator $\widehat\E_n$ is defined as \emph{stable} at $X=x$ (with respect to  a distance metric $d$) if
%\begin{equation} \label{eq:stability} \widehat\E_n\{ \widehat{f}(Z) \mid X=x\} - \widehat\E_n\{ {f}(Z) \mid X=x\} - \widehat\E_n\{ \widehat{b}(X) \mid X=x \} = o_\Pb(R_n^*(x)) \end{equation}
\begin{equation} \label{eq:stability}
\frac{\widehat\E_n\{ \widehat{f}(Z) \mid X=x\} - \widehat\E_n\{ {f}(Z) \mid X=x\} - \widehat\E_n\{ \widehat{b}(X) \mid X=x \}}{\sqrt{\E\left(\Big[ \widehat\E_n\{ f(Z) \mid X=x\} - \E\{f(Z) \mid X=x\} \Big]^2  \right)}} \ \inprob \ 0
\end{equation}
whenever  
$ d(\widehat{f},f) \inprob 0 $.
\end{definition}

\bigskip

Sample splitting plays an important role here:  the regression procedure as defined estimates the pseudo-outcome on a separate sample, independent from the one used in the second-stage regression via $\widehat\E_n$. Examples of such procedures can be found in subsequent sections, e.g., as illustrated in Figures \ref{fig:drlearner} and \ref{fig:rlearner}. The main role of sample splitting  is that it allows for informative error analysis while being agnostic about the first- and second-stage methods.  \\

\begin{remark}
With iid data, one can always obtain separate independent samples by randomly splitting the data in half (or in folds); further, to regain full sample size efficiency one can always swap the samples, repeat the procedure, and average the results, popularly called cross-fitting and used for example by \citet{bickel1988estimating}, \citet{robins2008higher}, \citet{zheng2010asymptotic}, and \citet{chernozhukov2018double}. In this paper, to simplify notation we always analyze a single split procedure, with the understanding that extending to an analysis of an average across independent splits is straightforward. \\
\end{remark}

As mentioned above, Definition \ref{def:stability} can be viewed as a generalization of the classic condition that 
$$ \frac{ (\Pn-\Pb)(\widehat{f}-f )}{1/\sqrt{n}} \inprob 0  $$
when $\| \widehat{f}-f\| \inprob 0$ (cf.\ \citet{andrews1994asymptotics}, Lemma 19.24 of \citet{van2000asymptotic}, Lemma 2 of \citet{zheng2010asymptotic}, term $\mathcal{I}_{3,k}$ of \citet{chernozhukov2018double}, Lemma 2 of \citet{kennedy2020sharp}, etc.). More specifically, Definition \ref{def:stability} generalizes  from settings involving averages $\Pn$ to generic regressions $\widehat\E_n$, where slower than root-n rates can appear in the denominator scaling. We also note that the stability in Definition \ref{def:stability} is local or pointwise, in the sense that it is defined at a particular $X=x$ (because pointwise rates are the focus of this paper). More global variants of stability are also possible; for example, after this work appeared on arxiv, \citet{rambachan2022counterfactual} used an $L_2$ version of stability (their Assumption B.1 and Proposition B.1) to study $L_2$ rates (their Lemma B.1). \\

The next result shows that generic linear smoothers are stable, with respect to a natural weighted $L_2$ distance. This and all other proofs are given  in Section \ref{sec:proofs}. Linear smoothers are a fundamental class of regression estimators and include, for example, linear regression, series methods, local polynomials, nearest neighbor matching, smoothing splines, kernel ridge regression, and some versions of random forests \citep{biau2012analysis, scornet2016random, verdinelli2021forest}. We suspect other kinds of estimators that are not linear smoothers are also stable in the sense of Definition \ref{def:stability}, but leave this to future work. \\

\begin{theorem} \label{thm:linearsmoothers}
Linear smoothers of the form $\widehat\E_n\{\widehat{f}(Z) \mid X=x\}=\sum_i w_i(x;X^n) \widehat{f}(Z_i)$ are stable in the sense of Definition \ref{def:stability}, with respect to distance 
$$ d(\widehat{f},f)  = \| \widehat{f}-f\|_{w^2} \equiv \sum_{i=1}^n \left\{ \frac{ w_i(x;X^n)^2}{\sum_j w_j(x;X^n)^2} \right\} \int \left\{ \widehat{f}(z) - f(z) \right\}^2 \ d\Pb(z \mid X_i) , $$
whenever $1/\| \sigma \|_{w^2} = O_\Pb(1)$ for $\sigma(x)^2 = \var\{f(Z) \mid X=x\}$. \\
\end{theorem}

We note that Theorem \ref{thm:linearsmoothers}  recovers results similar to Lemma 2 of \citet{kennedy2020sharp} for averages, i.e., when the weights all equal $w_i(x;X^n)=n^{-1}$, but also covers the much larger class of all linear smoothers. The norm $\| \cdot \|_{w^2}$ defining the relevant distance $d(\widehat{f},f)$ for linear smoothers is a natural weighted and conditional version of an $L_2(\Pb)$ norm. Namely, it is a weighted average of the conditional $L_2(\Pb)$ norm (given $X$), weighted more towards the point $X=x$ depending on how localized the weights $w_i(x;X^n)$ are. \\

The next result shows how stability and consistency of $\widehat{f}$ yield a rate of convergence result, relative to an oracle estimator that regresses the true unknown $f(Z)$ on $X$, and that a further small-bias condition yields oracle efficiency (i.e., asymptotic equivalence to the oracle estimator). This lays out a recipe for deriving convergence rates and conditions for oracle efficiency for general pseudo-outcome regression problems, which we will use in the CATE setup in the next section. \\

\begin{proposition} \label{prop:oracle}
Under the same setup from Definition \ref{def:stability}, also define:
\begin{enumerate}
%\item[$\boldsymbol\cdot$] 
\item[i.] $m(x) = \E\{f(Z) \mid X=x\}$ the conditional expectation of $f(Z)$ given $X$, 
\item[ii.] $\widehat{m}(x) = \widehat\E_n\{ \widehat{f}(Z) \mid X=x\}$ the regression of $\widehat{f}(Z)$ on $X$ in the test samples,  
\item[iii.] $\widetilde{m}(x) = \widehat\E_n\{ f(Z) \mid X=x\}$ the corresponding oracle regression of $f(Z)$ on $X$,
%\item $\widehat{b}(x) = \widehat{b}(x;D^n)  \equiv \E\{ \widehat{f}(Z) - f(Z) \mid D^n, X=x\}$ the conditional bias of the estimator $\widehat{f}$, 
%\item $\overline{b}(x) = \widehat\E_n\{ \widehat{b}(X) \mid X=x\}$ the regression of $\widehat{b}(X)$ on $X$ in the test samples, at $X=x$, 
\end{enumerate}
and the oracle mean squared error  $R_n^*(x)^2 = \E[\{ \widetilde{m}(x)-m(x)\}^2 ]$. 
If: 
\begin{enumerate}
\item the regression estimator $\widehat\E_n$ is stable with respect to distance $d$, and
\item $d(\widehat{f},f) \inprob 0$, 
\end{enumerate}
then 
$$\widehat{m}(x) - \widetilde{m}(x) = \widehat\E_n\{ \widehat{b}(X) \mid X=x\} + o_\Pb(R_n^*(x)). $$
Therefore if it further holds that
\begin{enumerate}
\item[3.] $\widehat\E_n\{ \widehat{b}(X) \mid X=x\}=o_\Pb(R_n^*(x))$,  
\end{enumerate}
then $\widehat{m}$ is oracle efficient, i.e., asymptotically equivalent to the oracle estimator $\widetilde{m}$ in the sense that 
$$ \frac{ \widehat{m}(x) - \widetilde{m}(x) }{R_n^*(x)} \ \inprob \ 0 . $$
\end{proposition}

\bigskip 

We now go on to discuss each component of Proposition \ref{prop:oracle}. The oracle risk $R_n^*(x)$ will typically be known (at least up to constants) based on structural assumptions on the target regression function $m(x)$ as well as the form of the estimator $\widehat\E_n$. For example, if $m$ is $s$-smooth  and $\widehat\E_n$ is an appropriate minimax optimal estimator (e.g., based on series or local polynomial regression) then $R_n^*(x) \asymp n^{-1/(2+d/s)}$. Similarly, if $m$ is $s$-sparse and $\widehat\E_n$ is an appropriate minimax optimal estimator (e.g., lasso) then $R_n^*(x) \asymp \sqrt{s \log d/n}$, under some assumptions. The stability condition was described previously; as mentioned there, it holds for linear smoothers and we expect it to hold more generally. The consistency condition $d(\widehat{f},f) \inprob 0$ is mild since it does not impose any particular rate requirement, and would generally follow from standard regression results (depending on the form of $\widehat{f}$); although consistency at any rate is sufficient for asymptotic negligibility, finite-sample behavior could depend on the rate at which $d(\widehat{f},f)$ converges to zero. \\

Under just stability of $\widehat\E_n$ and consistency of $\widehat{f}$, one obtains the rate of convergence result  
$$\widehat{m}(x) - m(x) = \widetilde{m}(x) - m(x) +  \widehat\E_n\{ \widehat{b}(X) \mid X=x\} + o_\Pb(R_n^*(x)) . $$
The bias term $\widehat\E_n\{ \widehat{b}(X) \mid X=x\}$ (recalling $\widehat{b}(x) = \E\{ \widehat{f}(Z) - f(Z) \mid D^n, X=x\}$) is the crucial piece that will generally determine whether the estimator $\widehat{m}$ is oracle efficient; if not, the rate would follow from this bias term. The bias  $\widehat{b}(x)$ can in be determined on a case-by-case basis, depending on the estimator $\widehat{f}$. For simple plug-in estimators of $f$, the bias would typically have a first-order dependence on the estimation error in whatever nuisance functions appear in $f$. However, in some cases one will be able to construct influence-function-based estimators of $f$ that have only a second-order dependence on nuisance estimation error. The availability of such estimators will depend on the form of $m$, but would be expected to exist when $m$  is a structured combination of regressions or densities (e.g., a dose-response curve, CATE, counterfactual density, etc.). We refer to Section 5.3 of \citet{kennedy2022semiparametric} for some related general discussion of using influence functions for non-pathwise differentiable quantities like $m$. \\

Note $\widehat\E_n\{ \widehat{b}(X) \mid X=x\}$ is actually an estimate of the regression of the bias term $\widehat{b}(X)$ at $X=x$; thus  it can be viewed as a smoothed bias, which may be more or less averaged or localized depending on the form of $\widehat\E_n$. Since the form of $\widehat{b}(x)$ will often be well-understood (e.g., derived analytically based on the form of $\widehat{f}$), in the following proposition we relate it to its smoothed analog $\widehat\E_n\{ \widehat{b}(X) \mid X=x\}$ in the linear smoother case. \\

\begin{proposition} \label{prop:smoothbias}
If 
 $\widehat{b}(x) = \widehat{b}_1(x) \widehat{b}_2(x)$, and 
 $\widehat\E_n$ is a linear smoother with $\sum_i |w_i(x;X^n)| = O_\Pb(c_n)$, 
then
$$ \widehat\E_n\{ \widehat{b}(X) \mid X=x\} = O_\Pb\left(  c_n \| \widehat{b}_1 \|_{w,p} \| \widehat{b}_2 \|_{w,q} \right)  $$
for the norm $\| f \|_{w,p}= \left[ \sum_i \left\{ \frac{|w_i(x;X^n)|}{\sum_j |w_j(x;X^n)|} \right\} | f(X_i)|^p \right]^{1/p}$ and $1/p + 1/q=1$ ($p,q>1$). 
\end{proposition}

\bigskip

Proposition \ref{prop:smoothbias} shows that the smoothed bias can be expressed  in terms of natural weighted norms of the components of the bias $\widehat{b}(x)$ itself. For many linear smoothers $\sum_i |w_i(x;X^n)| \leq C$ with probability one, so that $c_n=1$ (e.g., this is a condition in the celebrated theorem of \citet{stone1977consistent} guaranteeing weak universal consistency of linear smoothers). For series estimators with $w_i(x;X^n) = \frac{1}{n} \rho(x)^\T \widehat{Q}_{hx}^{-1} \rho(X_i)$ where $\rho$ is some basis of dimension $k$ and $\widehat{Q}_{hx}=\Pn(\rho\rho^\T)$, a simple bound can yield $c_n=\sup_x \sqrt{\rho(x)^\T \rho(x)}$ (which is of order $\sqrt{k}$ for many series). We suspect the dependence on $k$ can be avoided with more careful analysis, or when working with $L_2(\Pb)$ norms. \\

\section{DR-Learner} \label{sec:drlearner}

In this section we analyze a two-stage doubly robust estimator we refer to as the DR-Learner, following the naming scheme from \citet{nie2017quasi} and \citet{kunzel2019metalearners}. After detailing previous work, we describe the algorithm in detail, and then give model-agnostic error bounds which apply for arbitrary first-stage estimators, and as long as the second-stage estimator is stable in the sense of Definition \ref{def:stability}. We go on to apply the error bound in multiple nonparametric models incorporating smoothness or sparsity structure, and then explore the performance of the DR-Learner in simulations. \\

\subsection{Previous Work}

Variants of the DR-Learner have been used before, though often tied to particular estimators and not incorporating sample splitting, which can allow for model-agnostic error bounds and reduced bias. \citet{van2006statistical} (Section 4.2) appears to be the first to propose the general DR-Learner approach, i.e., flexible regressions of the pseudo-outcome \eqref{eq:pseudo} below on covariates. Specifically, \citet{van2006statistical} (along with \citet{van2013targeted} and \citet{luedtke2016super}) advocates regressing the pseudo-outcome on $V \subseteq X$ to construct candidate CATE estimators, and then selecting among them with a tailored cross-validation approach \citep{van2003cross}. The main distinction with our work is they did not give specific error guarantees for the CATE.  \\

To the best of our knowledge, previous papers  that do give specific error rates for the DR-Learner either employ stronger nuisance estimation conditions than we show are required, or else do not allow the CATE to be smoother or more structured than the individual regression functions $\mu_a$.  \citet{lee2017doubly} studied a local-linear version of the DR-Learner, but assumed the first-stage nuisance error was negligible. \citet{semenova2017estimation} and \citet{zimmert2019nonparametric} studied series and local-constant variants, respectively, but required conditions on nuisance estimation that are as restrictive as for the ATE. \citet{fan2019estimation} also studied a local-constant variant, but did not consider the case where the CATE is smoother than the regression functions. \\

Our results are closest to \citet{foster2019orthogonal}, who also considered the DR-Learner and gave an oracle inequality for generic empirical risk minimization when the loss involves complex nuisance functions. However there are some important distinctions to be made. First and most importantly, our error bound for the DR-Learner is doubly robust, involving a simple second-order product of nuisance errors. In contrast, the results of \citet{foster2019orthogonal} yield errors that are second-order but not doubly robust, instead involving $L_4$ errors of all nuisance components. This means our error bound is tighter when the propensity score and outcome regressions are estimated at different rates.  Second, we focus on local estimation error at a point, whereas \citet{foster2019orthogonal} consider global error (e.g., integrated MSE); in this sense our work is complementary. A third important distinction is that our results can be used to justify the validity of inferential procedures, such as confidence intervals and hypothesis testing, whereas \citet{foster2019orthogonal} focus on global rates of convergence. \\

There is also a related literature on regression with estimated pseudo-outcomes, similar in spirit to \citet{foster2019orthogonal}, though giving somewhat more specialized results. For example, \citet{ai2003efficient} and \citet{rubin2005general} both studied regression with general pseudo-outcomes, though they focused on particular estimators, and did not give doubly robust error bounds.  \citet{ai2003efficient} restricted their attention to sieves, and focused more on a finite-dimensional component appearing in the pseudo-outcome. \citet{rubin2005general} considered least squares, penalized methods, and linear smoothers in the second stage, but their error bounds are not doubly robust.\\

 \citet{kunzel2019metalearners} considered various method-agnostic ``meta-learners'', but not the DR-Learner; further, none of their methods are doubly robust, and so in general would inherit larger error rates from the underlying regression estimators. \\

\subsection{Construction \& Analysis}

The algorithm below describes our proposed construction of the DR-Learner. \\

\begin{algorithm}[DR-Learner] \label{alg:drlearner}
Let $(D_{1}^n,D_2^n)$ denote two independent samples of $n$ observations of $Z_i=(X_i,A_i,Y_i)$. 
\begin{enumerate}
\item[Step 1.] Nuisance training:
\begin{enumerate}
\item Construct estimates $\widehat\pi$ of the propensity scores $\pi$ using $D_{1}^n$. 
\item Construct estimates $(\widehat\mu_0,\widehat\mu_1)$ of the regression functions $(\mu_0,\mu_1)$ using $D_{1}^n$. 
\end{enumerate}
\item[Step 2.]  Pseudo-outcome regression: Construct the pseudo-outcome 
\begin{equation}
\widehat\varphi(Z) = \frac{A-\widehat\pi(X)}{\widehat\pi(X)\{1-\widehat\pi(X)\}} \Big\{ Y - \widehat\mu_A(X) \Big\} + \widehat\mu_1(X) - \widehat\mu_0(X) \label{eq:pseudo}
\end{equation}
and regress it on covariates $X$ in the test sample $D_2^n$, yielding
\begin{equation}
\widehat\tau_{dr}(x) = \widehat\E_n\{ \widehat\varphi(Z) \mid X=x \} .
\end{equation}
\item[Step 3.]  Cross-fitting (optional): Repeat Step 1--2,  swapping the roles of $D_1^n$ and $D_2^n$ so that $D_{2}^n$ is used for nuisance training and $D_{1}^n$ as the test sample. Use the average of the resulting two estimators as a final estimate of $\tau$. $K$-fold variants are also possible. 
\end{enumerate}
\end{algorithm}

\bigskip

 Figure \ref{fig:drlearner} gives a schematic illustrating the DR-Learner construction.  \\

\begin{figure}[h!]
\begin{center}
\fbox{\includegraphics[width=.8\textwidth]{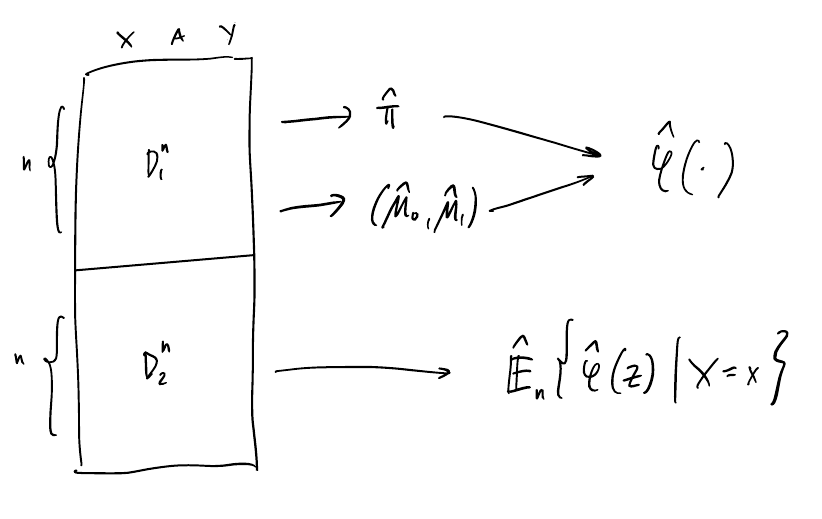}}
\caption{Schematic illustrating the DR-Learner approach. In the first stage, the nuisance functions $\widehat\pi$ and $(\widehat\mu_0,\widehat\mu_1)$ are estimated from training sample $D_{1}^n$. In the second stage, these estimates are used to construct an estimate of the pseudo-outcome $\widehat\varphi$, which is then regressed on $X$ using the  test sample $D_{2}^n$.} \label{fig:drlearner}
\end{center}
\end{figure}

\begin{remark}
The DR-Learner approach is motivated by the fact that \eqref{eq:pseudo} is the (uncentered) efficient influence function for the ATE \citep{robins1995semiparametric, hahn1998role}; this drives many of its favorable properties.  The intuition is that, to efficiently estimate the ATE, the standard doubly robust estimator  \emph{averages} the pseudo-outcome $\widehat\varphi$,  so to estimate the CATE the DR-Learner \emph{regresses} $\widehat\varphi$ on covariates. For a review of influence functions and semiparametric theory we refer to \citet{van2003unified}, \citet{tsiatis2006semiparametric}, and \citet{kennedy2022semiparametric}. \\
\end{remark}

Next we present the main result of this section, which gives error bounds  for the DR-Learner procedure (relative to an oracle) for arbitrary first-stage estimators, as long as the second-stage estimator is stable in the sense of Definition \ref{def:stability}. \\

\begin{theorem} \label{thm:drlearner}
Let $\widehat\tau_{dr}(x)$ denote the DR-Learner estimator detailed in Algorithm \ref{alg:drlearner}, which regresses the estimated pseudo-outcome $\widehat\varphi(Z)$ on covariates $X$. Assume:
\begin{enumerate}
\item The second-stage regression estimator $\widehat\E_n$ is stable with respect to distance $d$,
\item $d(\widehat\varphi,\varphi) \inprob 0$, 
\end{enumerate}
Let $\widetilde\tau(x) = \widehat\E_n\{ \varphi(Z) \mid X=x\}$ denote an oracle estimator that regresses the true pseudo-outcome $\varphi(Z)$ on $X$, and denote its risk by $R_n^*(x)^2=\E[ \{ \widetilde\tau(x) - \tau(x) \}^2 ]$. 
Then
$$\widehat\tau_{dr}(x) - \widetilde\tau(x) = \widehat\E_n\{\widehat{b}(X) \mid X=x\} + o_\Pb\Big( {R_n^*(x)} \Big) $$
for
$$ \widehat{b}(x) = \sum_{a=0}^1 \frac{ \{ \widehat\pi(x) - \pi(x) \} \{\widehat\mu_a(x) - \mu_a(x) \} }{a \widehat\pi(x) + (1-a) (1-\widehat\pi(x))}  $$
and $\widehat\tau_{dr}(x)$ is oracle efficient in the sense of Proposition \ref{prop:oracle} if $\widehat\E_n\{\widehat{b}(X) \mid X=x\}  = o_\Pb\Big( {R_n^*(x)} \Big)$. \\
\end{theorem}

The bound on the DR-Learner error given in Theorem \ref{thm:drlearner} shows that it can only deviate from the oracle error by at most a (smoothed) product of errors in the propensity score and regression estimators, thus allowing faster rates for estimating the CATE even when the nuisance estimates converge at slower rates. Importantly the result is agnostic about the methods used, and requires no special tuning or undersmoothing.  \\

Theorem \ref{thm:drlearner} gives a smaller risk bound compared to earlier work. 
\citet{semenova2017estimation} and \citet{zimmert2019nonparametric} gave a remainder error larger than the product of the nuisance mean squared errors, with oracle efficiency requiring this product to shrink faster than $1/nd$ and $1/(n/h)$ for series and kernel-based second stage regressions, respectively (for $h$ a shrinking bandwidth). \citet{fan2019estimation} assumed the CATE was only as smooth as the individual regression functions, giving a larger oracle risk. The results from \citet{foster2019orthogonal} are closest to ours, but as mentioned earlier their error bounds are not doubly robust, and instead involve $L_4$ errors of all nuisance components.  \\

\begin{remark}
Several important oracle inequalities for cross-validated selection of estimators exist in the literature \citep{van2003cross, luedtke2016super, diaz2018targeted}. These are relevant for cross-validated CATE estimation, but are conceptually different from our Theorem \ref{thm:drlearner}; for example, they use a different oracle. Namely, in cross-validated selection the oracle is the best performing among a group of learners, whereas in our setup the oracle is the specified learner $\widehat\E_n$ when it is given access to the true pseudo-outcome $\varphi(Z)$. \\
\end{remark}

\begin{remark}
In addition to showing double robustness and giving conditions for oracle efficiency, Theorem \ref{thm:drlearner} also has important implications for inference. Namely, if the DR-Learner is oracle efficient so that $\widehat\tau(x) - \tilde\tau(x) = o_\Pb(R_n^*(x))$, then whenever an inferential result is available for the oracle estimator $\tilde\tau(x)$ (which is just a standard regression of oracle pseudo-outcomes on covariates), this would also apply to $\widehat\tau(x)$, immediately allowing for the construction of confidence intervals, tests, etc. For example, if one used a local polynomial estimator for the second-stage regression $\widehat\E_n$, then under standard conditions  \citep{masry1996multivariate} $\tilde\tau(x)$ would be asymptotically normal, and so $\widehat\tau(x)$ would be as well, by virtue of the oracle equivalence. Thus standard confidence intervals could be constructed, treating the estimated pseudo-outcomes as one would usual observed outcomes in nonparametric regression. This is a benefit of the  asymptotic equivalence to the oracle in Theorem \ref{thm:drlearner}. In contrast, the results from \citet{foster2019orthogonal}, for example, instead give conditions under which the global risk of a DR-Learner is within a multiplicative constant of that of the oracle, which is not as immediately useful for inference. 
\end{remark}

\subsection{Examples \& Illustrations} \label{sec:drlex}

An important feature of the oracle result in Theorem \ref{thm:drlearner} is that it is essentially model-free: beyond a mild consistency assumption, it only requires that the second-stage regression estimator satisfies the stability condition in Definition \ref{def:stability}. In the following corollaries, we illustrate the flexibility of this result by applying it in settings where the nuisance functions and CATE are smooth or sparse (i.e., in settings where local polynomial, series, lasso, or random forest estimators would work well). Similar results could be obtained in generic models with known bounds on mean squared error rates. \\

\begin{corollary} \label{cor:drlsmooth}
Suppose the assumptions of Theorem \ref{thm:drlearner} hold. Further assume:
\begin{enumerate}
\item  $\widehat\E_n$ is a minimax optimal linear smoother with $\sum_i | w_i(x;X^n) | = O_\Pb(1)$.
\item The propensity score $\pi$ is $\alpha$-smooth, and $\| \widehat\pi - \pi \|_{w,2} = O_\Pb(n^{-1/(2+d/\alpha)})$.
\item The regression functions $\mu_a$ are $\beta$-smooth, and $\| \widehat\mu_a - \mu_a \|_{w,2} = O_\Pb(n^{-1/(2+d/\beta)})$.
\item The CATE $\tau$ is $\gamma$-smooth. 
\end{enumerate}
Then 
$$ \widehat\tau_{dr}(x) - \tau(x) = O_\Pb\left( n^{\frac{-1}{2+d/\gamma}} + n^{-\left(\frac{1}{2+d/\alpha} + \frac{1}{2 + d/\beta}\right)} \right)  $$
and the DR-Learner is oracle efficient if 
\begin{equation}
\sqrt{\alpha\beta} \geq   \frac{ d/2}{\sqrt{1 + \frac{d}{\gamma}\left(1 + \frac{d}{2\overline{s}} \right) }} , \label{eq:dreff1}
\end{equation}
where $\overline{s}=(\frac{\alpha^{-1}+\beta^{-1}}{2})^{-1}$ is the harmonic mean of $(\alpha,\beta)$.
\end{corollary}

\bigskip

Corollary \eqref{cor:drlsmooth} illustrates how the DR-Learner can adapt to smoothness in the CATE even when the propensity score and regression functions may be less smooth, and gives sufficient conditions for achieving the oracle rate  depending on nuisance smoothness and dimension $d$. \\

It is instructive to compare the sufficient condition in \eqref{eq:dreff1} to the analogous condition for root-n consistency of a standard doubly robust estimator of the average treatment effect, which is $\sqrt{\alpha\beta} \geq d/2$ (cf.\ Equation 25 of \citet{robins2009quadratic}). First, as the CATE smoothness gets larger, the sufficient condition \eqref{eq:dreff1} for the CATE approaches that for the ATE, as should be expected (intuitively, an infinitely smooth CATE should be nearly as easy to estimate as the ATE). Second, the term in the denominator  of \eqref{eq:dreff1}, dividing $d/2$, can be interpreted as a ``lowered bar'' for optimal estimation, due to the fact that the oracle rate $n^{\frac{-1}{2+d/\gamma}}$ is slower than root-n. This phenomenon was also noted in the dose-response estimation problem by \citet{kennedy2017nonparametric}, and is in contrast with the error bounds given in \citet{nie2017quasi} and \citet{zimmert2019nonparametric}, which required ATE-like conditions for oracle efficiency of CATE estimators, via $n^{-1/4}$ or faster rates on the nuisance estimators (note that if $\alpha=\beta$ then $\sqrt{\alpha\beta} \geq d/2$ means the nuisance error rate is faster than $n^{-1/4}$). In the next section we will show how the sufficient condition  \eqref{eq:dreff1} can even be improved upon. \\

Similar results could be obtained for other models. For example, suppose the propensity score and regression functions are $\alpha$- and $\beta$-sparse, respectively, and estimated at rates $\alpha \log d/2$ and $\beta \log d/n$, and similarly for the CATE with $\gamma$-sparsity. The corresponding condition for oracle efficiency would be $\alpha\beta \log^2 d / n^2 \leq \gamma \log d/n$, which is $\sqrt{\alpha\beta} \leq \gamma n / \log d$. However we note that we have not proved second-stage estimators attaining these rates satisfy the stability condition of Definition \ref{def:stability}. \\

\subsection{Simulation Experiments}

In this section we study some finite-sample properties  via simulations (R code is available in Section \ref{sec:appendix}). We use four methods for CATE estimation: a plug-in that estimates the regression functions $\mu_0$ and $\mu_1$ and takes the difference (called the T-Learner by \citet{kunzel2019metalearners}), the X-Learner from \citet{kunzel2019metalearners}, the DR-Learner from Section \ref{sec:drlearner}, and an oracle DR-Learner that uses the true pseudo-outcome in the second-stage regression. \\

First we use the piecewise polynomial model from the motivating example in Section \ref{sec:motiv}, with outcome and second-stage regressions  fit using \verb|smoothing.spline| in R. Figure \ref{fig:motiv} shows the mean squared error for the four CATE methods at $n=2000$ (based on 500 simulations with MSE averaged over 500 independent test samples), across a range of convergence rates for the propensity score estimator $\widehat\pi$. To control the convergence rate we constructed this estimator as $\widehat\pi=\expit\{\logit (\pi) + \epsilon_n\}$, where $\epsilon_n \sim N(n^{-\alpha},n^{-2\alpha})$ so that $\text{RMSE}(\widehat\pi) \sim n^{-\alpha}$ (the  estimator was truncated to lie within $[0.01,0.99]$ to ensure positivity holds). The results show that the plug-in estimator inherits the large error in estimating the individual regression functions, while the DR-Learner achieves much smaller errors and adapts to the smoothness of the CATE. The X-Learner has MSE in between the two. Consistent with Theorem \ref{thm:drlearner}, the MSE of the DR-Learner approaches that of the oracle as the propensity score estimation error decreases (i.e., as the convergence rate gets faster). \\

%\medskip

\begin{figure}[h!]
    \centering
    \begin{subfigure}[t]{0.49\textwidth}
        \centering
        \includegraphics[width=\linewidth]{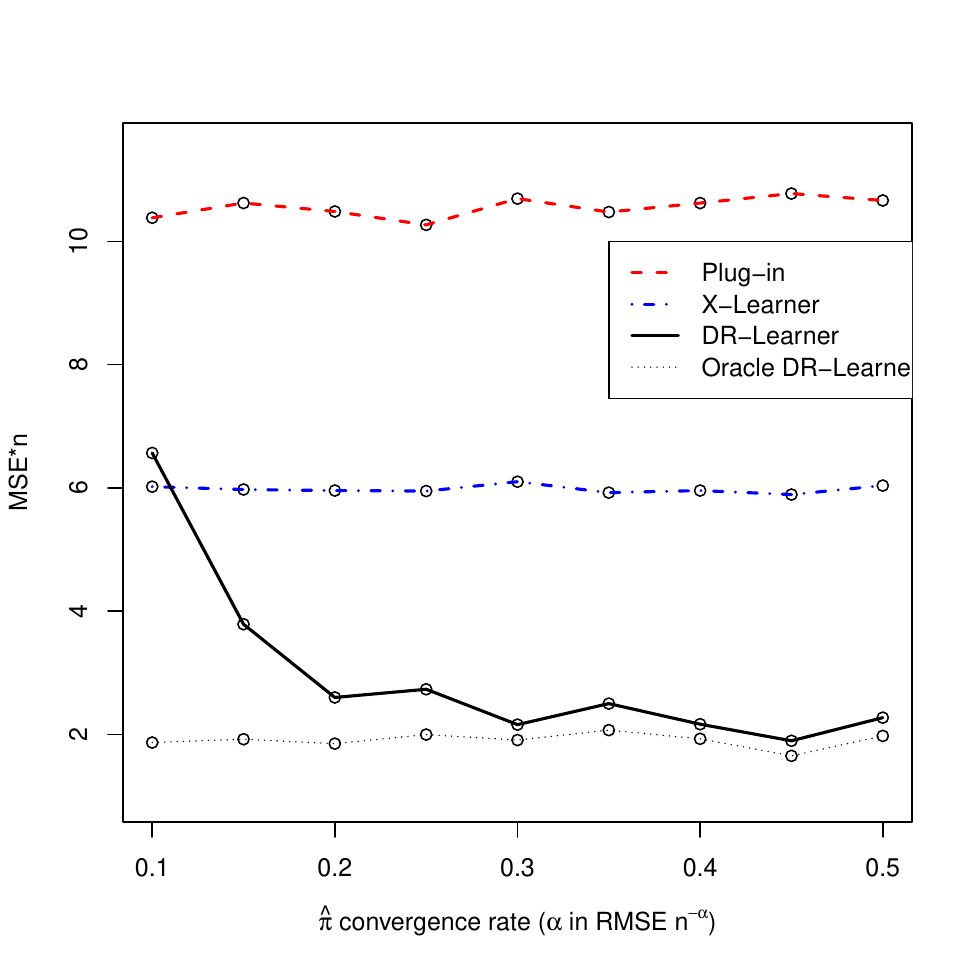} 
        \caption{Piecewise polynomial model from \ref{sec:motiv}} \label{fig:motiv}
    \end{subfigure}
    \hfill
    \begin{subfigure}[t]{0.49\textwidth}
        \centering
        \includegraphics[width=\linewidth]{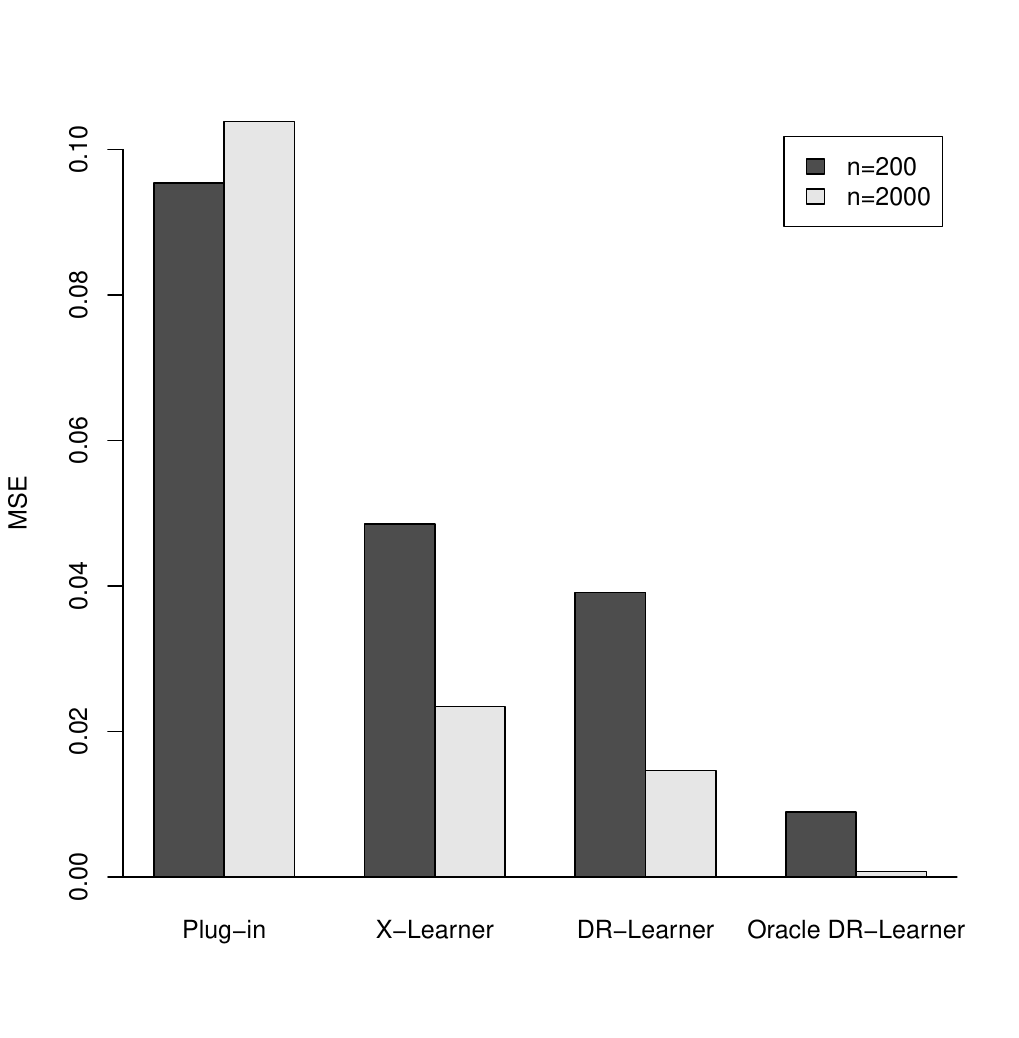} 
        \caption{High-dimensional model} \label{fig:lasso}
    \end{subfigure}
    \caption{Simulation results.}
\end{figure}

%\medskip

Next we consider a high-dimensional logistic model where $X=(X_1,...,X_d) \sim N(0,I_d)$, and 
$\logit\{\pi(x)\}   = \frac{1}{2\sqrt{\alpha}} \sum_{j=1}^\alpha x_j$, and $\logit\{ \mu_a(x) \} = \frac{1}{\sqrt{\beta}} \sum_{j=1}^\beta x_j $, 
so the propensity score and individual regressions are $\alpha$ and $\beta$ sparse, respectively, while the CATE is zero. The normalization of the  coefficients ensures $\pi(X) \in [0.2,0.8]$ with high probability and similarly for $\mu_a$.  We use standard cross-validation-tuned lasso for all model-fitting, i.e., to estimate the propensity scores, regression functions, and second-stage fits. Figure \ref{fig:lasso} shows mean squared errors for the four CATE methods when $d=500$ and $\alpha=\beta=50$, for $n=200$ and $n=2000$,  across 100 simulations (median of mean square errors is reported due to high skewness). As expected from Theorem \ref{thm:drlearner}, the DR-Learner is closer to the oracle than the plug-in or X-Learner, though this high-dimensional setup yields larger estimation error than the previous simulation model.  The relative performance of the DR-Learner seems to improve with sample size.

\section{Local Polynomial R-Learner} \label{sec:lprlearner}

The previous section gave general sufficient conditions under which the DR-Learner attains the oracle error rate of an estimator with direct access to the difference $Y^1-Y^0$, showing that this rate can in fact be achieved whenever the product of nuisance errors is of smaller order. This raises the crucial question of what happens when this product is not sufficiently small: in such regimes, is there any hope at still attaining the oracle error rate?  \\

The current section provides a first answer to this question, with a more refined analysis of a different estimator. Specifically, we analyze a double-sample-split local polynomial adaptation of the R-Learner \citep{nie2017quasi, robinson1988root}, which we call the lp-R-Learner for short. The R-Learner of \citet{nie2017quasi} is a nonparametric RKHS regression-based extension of the double-residual regression method of \citet{robinson1988root}.  A nonparametric series-based version of the R-Learner was also proposed in Example 4 of \citet{robins2008higher}, though assuming known propensity scores and not incorporating outcome regression. \citet{chernozhukov2017orthogonal} studied a lasso version of the R-learner. Some important distinctions between our results and those in previous work include the following: (i) our estimator is built from local polynomials, and incorporates a specialized form of sample splitting inspired by \citet{newey2018cross} for bias reduction, (ii) our sufficient conditions for attaining oracle efficiency are substantially weaker than the $n^{-1/4}$ rates in \citet{nie2017quasi} and \citet{chernozhukov2017orthogonal}, and (iii) we give specific rates of convergence outside the oracle regime. \\

We first describe the lp-R-Learner in detail, then give the main error bound result, which holds under a H\"{o}lder-smooth model, and is valid for a wide variety of tuning parameter choices. Following that, we optimize the bound with specific tuning parameter choices, under different sets of conditions, and discuss the resulting rates. \\

\subsection{Construction}

The algorithm below describes the lp-R-Learner construction. \\

\begin{algorithm}[lp-R-Learner] \label{alg:lprlearner}
Let $(D_{1a}^n,D_{1b}^n,D_2^n)$ denote three independent samples of $n$ observations of $Z_i=(X_i,A_i,Y_i)$. \\

Let $b: \R^d \mapsto \R^p$ denote the vector of basis functions consisting of all powers of each covariate, up to order $\lfloor\gamma\rfloor$, and all interactions up to degree $\lfloor\gamma\rfloor$ polynomials (cf.\ \citet{masry1996multivariate}). Let $K_{hx}(X)=\frac{1}{h^d}K\left( \frac{X-x}{h} \right)$ for $K: \R^d \mapsto \R$ a bounded kernel function with support $[-1,1]^d$, and $h$ a bandwidth parameter.
\begin{enumerate}
\item[Step 1.] Nuisance training:
\begin{enumerate}
\item Using $D_{1a}^n$, construct estimates $\widehat\pi_a$ of the propensity scores $\pi$. 
\item Using $D_{1b}^n$, construct estimates $\widehat\eta$ of the regression function $\eta=\pi \mu_1 + (1-\pi) \mu_0$, and estimates $\widehat\pi_b$ of the propensity scores $\pi$. 
\end{enumerate}
\item[Step 2.]  Localized double-residual regression:\\

Define $\widehat\tau_{r}(x)$ as the fitted value from a kernel-weighted least-squares regression (in the test sample $D_2^n$) of outcome residual $(Y-\widehat\eta)$ on basis terms $b$ scaled by the treatment residual $(A-\widehat\pi_b)$, with weights $\left(\frac{A-\widehat\pi_a}{A-\widehat\pi_b}\right) K_{hx}$. Thus $\widehat\tau_r(x)=b(0)^\T \widehat\theta$ for
\begin{equation}
\widehat\theta =  \argmin_{\theta \in \R^p} \ \Pn\left(  K_{hx}(X) \left\{ \frac{ A - \widehat\pi_a(X) }{A-\widehat\pi_b(X)} \right\} \left[ \Big\{ Y - \widehat\eta(X) \Big\} -  \theta^\T b(X-x) \Big\{ A - \widehat\pi_b(X) \Big\} \right]^2 \right). \label{eq:rlearner}
\end{equation}
\item[Step 3.]  Cross-fitting (optional): Repeat Step 1--2 twice, first using $(D_{1b}^n,D_2^n)$ for nuisance training and $D_{1a}^n$ as the test sample, and then  using $(D_{1a}^n,D_2^n)$ for training and $D_{1b}^n$ as the test sample. Use the average of the resulting three estimators of $\tau$ as the final estimator $\widehat\tau_r$. 
\end{enumerate}
\end{algorithm}

\bigskip

\begin{remark}
The kernel weights in the second step regression need to be multiplied by the ratio $(A-\widehat\pi_a)/(A-\widehat\pi_b)$ in order to ensure the independence of relevant products of nuisance estimators (i.e., that they are built from separate samples $D_{1a}^n$ and $D_{1b}^n$). This allows for multiplicative biases and thus faster rates due to undersmoothing, first introduced by \citet{newey2018cross} but for $\sqrt{n}$-estimable functionals. In other words, this ensures the bias of the lp-R-Learner  equals a product of biases of the nuisance estimators; other, different ratios used in this kernel weight would therefore generally not work in the same way.    \\
\end{remark}

 Figure \ref{fig:rlearner} gives a schematic illustrating the lp-R-Learner construction. \\

\begin{figure}[h!]
\begin{center}
\fbox{\includegraphics[width=.92\textwidth]{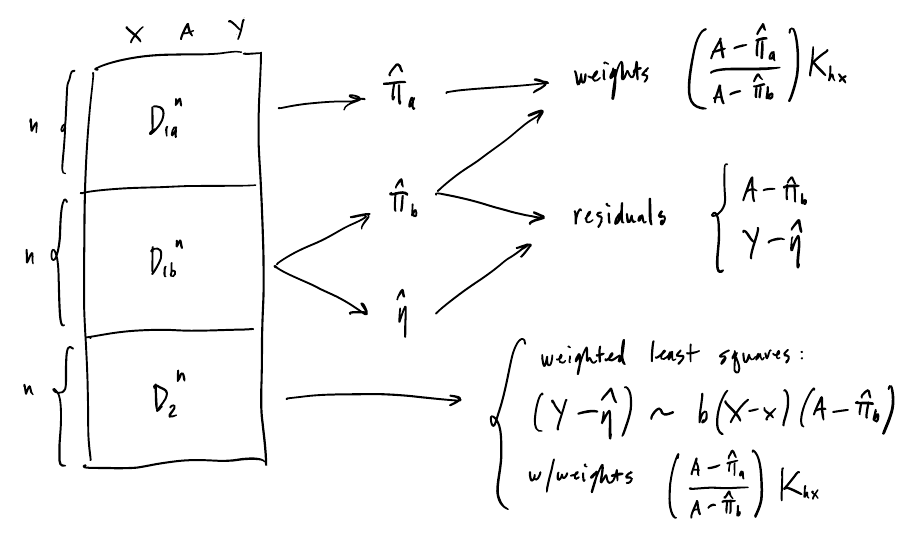}}
\caption{Schematic illustrating the lp-R-Learner approach. In the first stage, the nuisance functions $\widehat\pi_a$ and $(\widehat\pi_b,\widehat\eta)$ are estimated from training samples $D_{1a}^n$ and $D_{1b}^n$, respectively. In the second stage, these estimates are used in a kernel-weighted least squares regression of residuals $(Y-\widehat\eta)$ on residual-scaled basis terms $(A-\widehat\pi_b)b$, with weights $\left(\frac{A-\widehat\pi_a}{A-\widehat\pi_b}\right) K_{hx}$.} \label{fig:rlearner}
\end{center}
\end{figure}

\subsection{Main Error Bound \& Oracle Results}

Before giving the main error bound in this section, we first present the following condition on the nuisance estimators that we use in the analysis. At a high level, this condition requires the nuisance estimators to be linear smoothers with particular bias and variance bounds. \\

\begin{condition} \label{con:nuiscon}
The nuisance estimators $(\widehat\pi_a,\widehat\pi_b,\widehat\eta)$ are linear smoothers of the form 
\begin{align}
\widehat\pi_j(x) &= \sum_{i \in D_{1j}^n} w_{i\alpha}(x;X_{1j}^n) A_i \tag{1a} \label{eq:nuisfrm}  \\% \ \text{ and } \  \ 
\widehat\eta(x) &= \sum_{i \in D_{1b}^n} w_{i\beta}(x;X_{1b}^n) Y_i  \nonumber
\end{align}
with weights $w_{i\cdot}(x;X_{1\cdot}^n)$ depending on tuning parameter $k$, which are localized in the sense that 
\begin{align}
w_{i\cdot}(x;X_{1\cdot}^n)=0 \text{ whenever } \| X_i-x \|>1/k^{1/d} \label{eq:nuiswts} \tag{1b}
\end{align}
and which satisfy the conditional bias and variance bounds
\begin{align}
&\left| \E\{ \widehat\pi_j(x) \mid X_{1j}^n\} - \pi(x) \right| \lesssim k^{-\alpha/d} & \var\{ \widehat\pi_j(x) \mid X_{1j}^n\} \lesssim k/n \nonumber \\
&\left| \E\{ \widehat\eta(x) \mid X_{1b}^n\} - \eta(x) \right| \lesssim k^{-\beta/d} & \var\{ \widehat\eta(x) \mid X_{1b}^n\} \lesssim k/n. \label{eq:nuisbias} \tag{1c}
\end{align}

\end{condition}

\bigskip

Conditions \eqref{eq:nuisfrm}--\eqref{eq:nuisbias} are relatively standard. Many popular estimators take the form given in \eqref{eq:nuisfrm}, as discussed just before Theorem \ref{thm:linearsmoothers}. Condition \eqref{eq:nuiswts} holds for several prominent linear smoothers. For example, it holds for series estimators built from $k$ basis terms, if properly localized, and for standard kernel or local polynomial estimators when taking the bandwidth parameter as $h \sim k^{-d}$, as shown for example in Proposition 1.13 of \citet{tsybakov2009introduction}. \\

Condition \eqref{eq:nuisbias} also has been shown to hold for series and local polynomial estimators, for example, when the underlying regression function is appropriately smooth. In particular, under standard conditions, the bias part of \eqref{eq:nuisbias} would hold for these methods when the propensity score $\pi$ is $\alpha$-smooth and the regression function $\eta$ is $\beta$-smooth; we again refer to \citet{belloni2015some} and \citet{tsybakov2009introduction} for a review of related results.  For \eqref{eq:nuisbias} to hold uniformly over all $x \in \mathcal{X}$ would typically mean these bounds would only hold up to log factors; however our result will only require the bias to be controlled locally, near the point at which the CATE is to be estimated. \\

The next result gives error bounds on the lp-R-Learner from Algorithm \ref{alg:lprlearner}, under H\"{o}lder smoothness conditions. \\

\begin{theorem} \label{thm:lprlearner}
Let $\widehat\tau_r(x)$ denote the lp-R-Learner estimator detailed in Algorithm \ref{alg:lprlearner}. Assume: 

\begin{enumerate}
\item The estimator $\widehat\eta$ and observations $Z$ are bounded, and $X$ has density bounded above.  \label{ass:zestbd}
\item The estimators $\widehat\pi_j$ satisfy $\Pb\{\epsilon \leq \widehat\pi_j(x) \leq 1-\epsilon\}=1$ for some $\epsilon>0$. \label{ass:psbd}
\item The eigenvalues of the sample Gram matrices $\widehat{Q}_{hx}$ and $\widetilde{Q}_{hx}$ defined in \eqref{eq:qtildehat} are bounded away from zero in probability.  \label{ass:eigen}
\item The  nuisance estimators $(\widehat\pi_a,\widehat\pi_b,\widehat\eta)$ satisfy Condition \ref{con:nuiscon}, with the bias and variance bounds holding for all $x'$ such that $\| x' - x\| \leq h$. \label{ass:nuiserr}
\end{enumerate}
\medskip

Let $s=\frac{\alpha+\beta}{2}$ denote the average smoothness of the propensity score and regression function. 
Then, if the CATE $\tau(x)$ is $\gamma$-smooth and $k^{-(\alpha \wedge \beta)/d} \lesssim k/n$, we have
\begin{align*}
\widehat\tau_r(x) - \tau(x) &=  O_\Pb\left( h^{\gamma} + k^{-2s/d} + k^{-2\alpha/d} + \frac{1}{\sqrt{nh^d}}\left( 1 + \frac{k}{n} \right)  \right) .
\end{align*}
\end{theorem}

\bigskip

Before detailing the result and implications of Theorem \ref{thm:lprlearner}, we first discuss the assumptions. The first part of Assumption \ref{ass:zestbd} is mostly to simplify presentation, and could be weakened at the cost of added complexity; the second part  ensuring $X$ has bounded density is more crucial, but still mild. Assumption \ref{ass:psbd} is standard in the causal literature, and in theory could be guaranteed by simply thresholding propensity score estimates; however, extreme propensity values (i.e., positivity violations) are an important issue in practice, especially in the nonparametric and high-dimensional setup \citep{d2017overlap}.  Assumption \ref{ass:eigen} is relatively standard (see, e.g., Assumption (LP1) of \citet{tsybakov2009introduction}) but would restrict how $n$ and $h$ scale; when $d$ is fixed, standard bandwidth choices should suffice. Assumption \ref{ass:nuiserr} is arguably most crucial, and does the most work in the proof (together with the specialized sample splitting); however, as detailed in the discussion of  Condition \ref{con:nuiscon} prior to the theorem statement, Assumption \ref{ass:nuiserr} uses standard conditions commonly found in the nonparametric regression literature. \\

Now we give some discussion and interpretation of the (in-probability) error bound of Theorem \ref{thm:lprlearner}. The first three terms are the bias, and the last two are the variance (on the standard deviation scale). The bias has three components. The first $h^\gamma$ bias term comes from the bias of an oracle estimator with access to the true propensity score and regression function, and matches the bias of an oracle with direct access to the difference $Y^1-Y^0$. The other two bias terms come from nuisance estimation: the first $k^{-2s/d}$ term is the product of the biases of the propensity score and regression estimators, whereas the second $k^{-2\alpha/d}$ term is the squared bias of the propensity score estimator. If the propensity score is at least as smooth as the regression function, then the first $k^{-2s/d}$ term will dominate. Some heuristic intuition about why these specific bias terms arise is as follows: by virtue of its least squares construction, the lp-R-Learner can be viewed as a product of the inverse of an ``$X^\T X$-like'' term involving products of $\widehat\pi_a$ and $\widehat\pi_b$, and an ``$X^\T Y$-like'' term involving products of $\widehat\pi_a$ and $\widehat\eta$. 
The variance has two components. As with the bias, the first $1/\sqrt{nh^d}$ term is the standard deviation of an oracle estimator with access to the true nuisance functions. The second term is a product of this oracle standard error with $k/n$, which is the additional variance coming from nuisance estimation (in fact $k/n$ is the product of the standard deviations of the nuisance estimators). In standard regimes the tuning parameter $k$ would have to be chosen of smaller order than $n$ (e.g., $k\log k/n \rightarrow 0$ as in \citet{belloni2015some} and elsewhere) in order for Condition \ref{con:nuiscon} to hold, making the variance contribution from nuisance estimation asymptotically negligible. This last point will be discussed further shortly. \\

Now we give conditions under which the oracle rate is achieved by the lp-R-Learner, which we conjecture are not only sufficient but also necessary conditions. \\

\begin{corollary} \label{cor:lproracle}
Suppose the assumptions of Theorem \ref{thm:lprlearner} hold. Further assume $\alpha \geq \beta$ so the propensity score is smoother than the regression function, and take
\begin{enumerate}
\item $h \sim n^{-1/(2\gamma+d)}$, and
\item $k \sim n/ \log^2 n$ so that $k \log k/n \rightarrow 0$.
\end{enumerate}
Then, up to log factors,
$$ \widehat\tau_r(x) - \tau(x) = O_\Pb\left( n^{-\gamma/(2\gamma+d)} + n^{-2s/d} \right) $$
and the oracle rate is achieved if the average nuisance smoothness satisfies $s \geq \frac{d/4}{1 + d/2\gamma}$.
\end{corollary}

\bigskip

Corollary \ref{cor:lproracle} shows that an undersmoothed lp-R-Learner can be oracle efficient under weaker conditions than the error bound given in Theorem \ref{thm:drlearner} would indicate for the DR-Learner. \\

\begin{remark}
Note taking $k \sim n/\log^2 n$ amounts to undersmoothing as much as possible: it drives down nuisance bias, letting the variance $k/n$ go to zero only very slowly, since the contribution of the latter is asymptotically negligible for CATE estimation as long as $k \lesssim n$. In general, undersmoothing can often require specific knowledge of smoothness parameters, and data-driven approaches  are often lacking. However, it is worth noting that in Corollary \ref{cor:lproracle} the choice of $k$ only depends on the sample size, and not on  the underlying nuisance smoothness, for example, making it more feasible to implement. For example, one could set $k$ and then use cross-validation or other approaches to select $h$ \citep{van2003cross,bibaut2017data}; however we leave a formal study of this to future work.  \\
\end{remark}

The result in Corollary \ref{cor:lproracle} is remniscient of a similar phenomenon in marginal ATE estimation, where \citet{robins2009semiparametric} showed that the condition $s \geq d/4$ is necessary and sufficient for the existence of root-n consistent estimators of the ATE functional $\E\{ \E(Y \mid X,A=1)\}$. Our result thus shows that $s \geq d/4$ is also sufficient for oracle efficient estimation of the CATE, but now the oracle rate is $n^{-\gamma/(2\gamma+d)}$ rather than root-n, and so there is in fact a weaker bar for oracle efficiency (namely $s \geq \frac{d/4}{1+d/2\gamma}$), depending on the smoothness $\gamma$. At one extreme, when the CATE is infinitely smooth, the condition we give for oracle efficiency of the lp-R-Learner recovers the usual $s \geq d/4$ condition for the ATE. At the other extreme, when the CATE is non-smooth, oracle efficiency can be much easier to achieve (e.g., if the CATE is only $\gamma=1$-smooth, it is only required that $s \geq 1/2$ even for arbitrarily large dimension $d$). We conjecture that the condition $s \geq \frac{d/4}{1+d/2\gamma}$ is not only sufficient but may also be necessary for oracle effiency in the above H\"{o}lder model, making the proposed lp-R-Learner minimax optimal in this regime when $\alpha \geq \beta$; however, we leave a proof of this to future work. \\

When $s < \frac{d/4}{1+d/2\gamma}$ and the oracle rate is not achieved, the rate $n^{-2s/d}$ is slower than the usual functional estimation rate $n^{-4s/(4s+d)}$, which is minimax optimal for the ATE if the covariate density is smooth enough \citep{robins2009semiparametric}, and also occurs for simpler functionals like the expected density \citep{bickel1988estimating, birge1995estimation}. To illustrate this gap, if $s=d/8$ then the rate in Corollary \ref{cor:lproracle} is $n^{-1/4}$ while the usual functional minimax rate is $n^{-1/3}$. \\

\begin{remark}
There is  a trade-off between the DR-Learner and lp-R-Learner. In short, the DR-Learner provides somewhat more general  guarantees, and is computationally more straightforward, while the lp-R-Learner can achieve faster rates in somewhat more specialized Holder smoothness models, but is also somewhat more difficult to compute in practice, and requires undersmoothing.  \\
\end{remark}

\subsection{Faster Rates with Known Covariate Density} \label{sec:lprfast}

Here we briefly  consider how the lp-R-Learner rates from Corollary \ref{cor:lproracle} can be improved by exploiting structure in the covariate density, as in \citet{robins2008higher,robins2017minimax}. This is somewhat paradoxical since the CATE itself does not depend on the covariate density. \\

Some intuition for a possible rate improvement is as follows. If the density of the covariates $X$ was known, then one could construct nuisance estimators $(\widehat\pi_a,\widehat\pi_b,\widehat\eta)$ satisfying Condition \ref{con:nuiscon} without any restriction on the tuning parameter $k$, such as the $k\log k/n \rightarrow 0$ restriction employed in Corollary \ref{cor:lproracle}. For example, a series estimator $\widehat\eta$ could satisfy Condition \ref{con:nuiscon} for any choice of $k$, if it took the form
\begin{equation}
\widehat\eta(x) = b(x)^\T \Big[ \E\left\{b(X) b(X)^\T \right\} \Big]^{-1} \Pn\{ b(X) Y\}
\end{equation}
where $b$ is a vector of appropriate basis functions (different from those used in the lp-R-Learner construction), and similarly for $\widehat\pi$. Intuitively this is because restrictions like $k\log k/n \rightarrow 0$ are only required to ensure the inverse of the sample Gram matrix $\Pn(bb^\T)$ converges to a limit (in operator norm), whereas if the density of $X$ was known, then one could simply compute the population Gram matrix $\E(bb^\T)$ exactly. \\

We conjecture however that the $n^{-2s/d}$ rate from Corollary \ref{cor:lproracle} may not be improvable in the non-random fixed $X$ case. A fixed design setup was recently considered by \citet{gao2020minimax}, though in a somewhat specialized model where the propensity scores have zero smoothness. In fact when $\alpha=0$ our rate matches their lower bound in the non-oracle regime.  \\

The next result gives rates for the lp-R-Learner when there are no restrictions on the choice of nuisance tuning parameter $k$. \\

\begin{corollary} \label{cor:lprknowndens}
Suppose the assumptions of Theorem \ref{thm:lprlearner} hold, and in particular suppose Assumption \ref{ass:nuiserr} holds without any restrictions on $k$ (e.g., as if the density of $X$ is known). \\

Then if $\alpha \geq \beta$ the convergence rate in Theorem \ref{thm:lprlearner} is optimized by taking
$$ h \sim n^{\frac{-3s}{2s\gamma + (s+\gamma)d}} \ \text{ and } \ 
k \sim n^{\frac{3\gamma d/2}{2s\gamma + (s+\gamma)d}} $$
which gives
$$ \widehat\tau(x) - \tau(x) = O_\Pb\left( n^{-\gamma/(2\gamma+d)} + n^{\frac{-3s}{2s+d(1+s/\gamma)}} \right) , $$
and so the oracle rate is achieved if the average nuisance smoothness satisfies $s \geq \frac{d/4}{1 + d/2\gamma}$.
\end{corollary}

\bigskip

\begin{remark}
When the density of the covariates $X$ is unknown but smooth and estimable at fast enough rates, we expect results similar to those in Corollary \ref{cor:lprknowndens} to still hold. This phenomenon also occurs for the ATE functional \citep{robins2008higher,robins2017minimax}. However, here the analysis of the lp-R-Learner becomes much more complicated, so we leave this to future work. \\
\end{remark}

When the nuisance tuning parameter $k$ is unrestricted, one needs to balance all three dominant terms from Theorem \ref{thm:lprlearner}: the two bias terms $h^\gamma$ and $k^{-2s/d}$, as well as the variance term $\frac{k/n}{\sqrt{nh^d}}$. Notice the ``elbow'' at $s \geq \frac{d/4}{1 + d/2\gamma}$ occurs here even when the density is known, perhaps again suggesting that this condition for oracle efficiency may be both sufficient and necessary. Whether the rate $n^{\frac{-3s}{2s+d(1+s/\gamma)}}$ is minimax optimal or not in the non-oracle regime is unknown; we will pursue this in future work. \\

Note the rate $n^{\frac{-3s}{2s+d(1+s/\gamma)}}$ is slightly slower than the usual functional estimation rate $n^{-4s/(4s+d)}$. For example, when $\gamma \rightarrow \infty$ and $s=d/8$, the CATE rate from Corollary \ref{cor:lprknowndens} is $n^{-3/10}$ whereas the usual functional estimation rate is $n^{-1/3}$. Thus there do appear to be  benefits of exploiting structure in the covariate density, but whether the gap between the improved rate of Corollary \ref{cor:lprknowndens} and the usual rate can be closed is  unclear. An interesting more  philosophical question is whether structure in the covariate density should be exploited for CATE estimation in practical settings, since the CATE itself does not depend on the covariate density. \\

\section{Discussion}

In this paper we studied the problem of estimating conditional  treatment effects, giving new results that apply to a wider variety of methods and that better exploit when the CATE itself is structured, compared to the current state-of-the-art. Sections \ref{sec:genoracle} and \ref{sec:drlearner} were more practically oriented, giving model-free yet informative error bounds for general regression with estimated outcomes, and the DR-Learner method for CATE estimation, along with examples from smooth and sparse models, illustrating the flexibility of Proposition \ref{prop:oracle} and Theorem \ref{thm:drlearner}. In contrast, Section \ref{sec:lprlearner} was more theoretically oriented, aiming instead at understanding the fundamental statistical limits of CATE estimation, i.e., the smallest possible achievable error. We derived upper bounds on this error with a specially constructed (and tuned) estimator called the lp-R-Learner. Namely we showed that in a H\"{o}lder model, oracle efficiency is possible under weaker conditions on the nuisance smoothness than indicated from the DR-Learner error bound given in Theorem \ref{thm:drlearner} -- which itself is an improvement over conditions given in previous work. \\

Figure \ref{fig:rates} summarizes the various error rates in this paper graphically, in an illustrative setup where the dimension is $d=20$, CATE smoothness is $\gamma=2d$, and nuisance smoothness $\alpha=\beta=s$ varies from 0--10. This shows the various gaps between methods/rates, in an interesting regime where the CATE is relatively smooth compared to the nuisance functions. \\

\begin{figure}[h!]
\begin{center}
{\includegraphics[width=.6\textwidth]{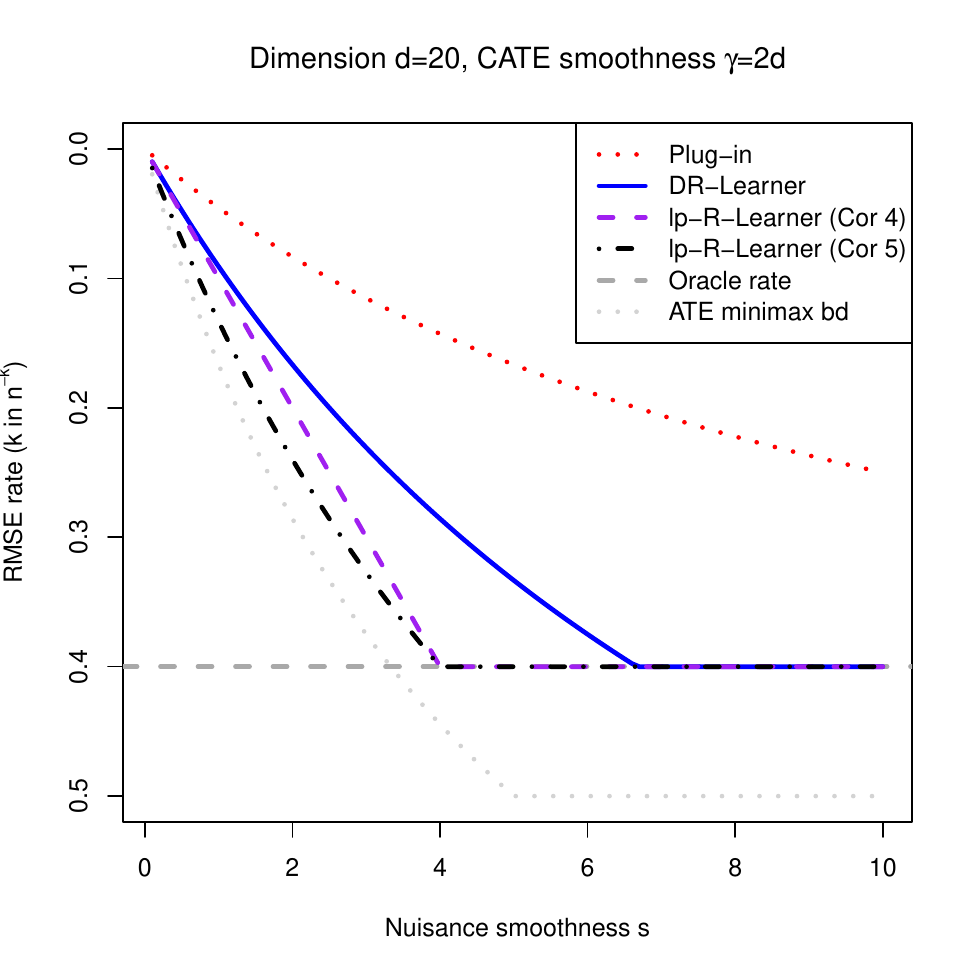}}
\caption{Illustration of error bounds in this work, as a function of nuisance smoothness $s=\alpha=\beta$. The dotted gray line represents the classical minimax lower bound for functional estimation, which is $\sqrt{n}$ when $s \geq d/4$. The dashed gray line is the oracle rate that would be achieved by a minimax optimal regression using $Y^1-Y^0$. The red line is the bound for a plug-in estimator, which is just the rate for estimating the individual regression functions. The blue line is the rate for the DR-Learner given in Theorem \ref{thm:drlearner}, which matches the oracle under conditions given there. The purple line is the rate achieved by the lp-R-Learner when tuned as in Corollary \ref{cor:lproracle}, which matches the oracle when $s \geq \frac{d/4}{1+d/2\gamma}$. The black line is the improved rate achieved by the lp-R-Learner when tuned as in Corollary \ref{cor:lprknowndens}, e.g., with known covariate density.} \label{fig:rates}
\end{center}
\end{figure}

Our work raises some interesting open questions, some of the more immediate of which we list here for reference: 

\medskip

\begin{enumerate}
\item Can the error bounds in Proposition \ref{prop:oracle} and Theorem \ref{thm:drlearner} be improved without committing to particular first- or second-stage methods? 
\item What rates can be achieved by specialized sample splitting and tuning of a DR-Learner, rather than an R-Learner?
\item What are the analogous results of Corollaries \ref{cor:drlsmooth} for other function classes?
\item Can the error bounds in Theorem \ref{thm:lprlearner} be obtained with other methods? A natural alternative is a higher-order influence function approach \citep{robins2008higher,robins2017minimax}, which could avoid the $\alpha \geq \beta$ condition, perhaps at the cost of some extra complexity.
\end{enumerate}

\bigskip

We have obtained partial answers to some of these questions, but most remain unanswered and left for future work. One of the deepest open questions is whether the rates given in Corollaries  \ref{cor:lproracle} and  \ref{cor:lprknowndens} are minimax optimal or not (in the fixed and random design setups, respectively).   \\

\section*{Acknowledgements}

This research was supported by NSF DMS Grant 1810979, NSF CAREER Award 2047444, and NIH R01 Grant LM013361-01A1.The author thanks Sivaraman Balakrishnan, Matteo Bonvini, Aaron Fisher, Virginia Fisher, Jamie Robins, and Larry Wasserman for very helpful discussions.  \\

\section*{References}
\vspace{-1cm}
\bibliographystyle{abbrvnat}
\bibliography{/Users/kennedye/Desktop/flashdrive/research/bibliography}

\pagebreak

\section{Proofs}
\label{sec:proofs}

\subsection{Proof of Theorem \ref{thm:linearsmoothers}}

Here we use the notation of Proposition \ref{prop:oracle}.
Letting $T_n = \widehat{m}(x) - \widetilde{m}(x) - \widehat\E_n\{\widehat{b}(X) \mid X=x\}$  denote the numerator of the left-hand side of \eqref{eq:stability}, we will show that
$$ T_n = O_\Pb\left( \frac{\| \widehat{f}-f \|_{w^2}}{\|\sigma \|_{w^2}} R_n^*(x) \right) $$
which yields the result when $1/\|\sigma\|_{w^2} = O_\Pb(1)$. 
First note that for linear smoothers we have
$$ T_n= \widehat\E_n\{ \widehat{f}(Z) - f(Z) - \widehat{b}(X) \mid X=x\} = \sum_{i=1}^n w_i(x;X^n) \Big\{ \widehat{f}(Z_i) - f(Z_i) - \widehat{b}(X_i) \Big\}  $$
and this term has mean zero since
\begin{align*}
\E\Big\{ \widehat{f}(Z_i) - f(Z_i) - \widehat{b}(X_i) \mid D^n, X^n \Big\} = \E\Big\{ \widehat{f}(Z_i) - f(Z_i) - \widehat{b}(X_i) \mid D^n, X_i \Big\}  = 0
\end{align*}
by definition of $\widehat{b}$ and iterated expectation. Therefore
 \begin{align}
 \E&(T_n^2 \mid D^n, X^n ) = \var\left[ \sum_{i=1}^n w_i(x;X^n)  \Big\{ \widehat{f}(Z_i) - f(Z_i) - \widehat{r}(X_i) \Big\} \Bigm| D^n, X^n \right] \nonumber \\
& =  \sum_{i=1}^n w_i(x;X^n)^2 \ \var\Big\{ \widehat{f}(Z_i) - f(Z_i) \mid D^n, X_i \Big\}  \leq \| \widehat{f} - f \|_{w^2}^2 \sum_{i=1}^n w_i(x;X^n)^2 \label{eq:lsep1}
 \end{align}
 where the second line follows since $\widehat{f}(Z_i) - f(Z_i)$ are independent given the training data, and the third since $\var(\widehat{f}-f \mid D^n,X) \leq \E\{(\widehat{f}-f)^2 \mid D^n, X\}$ and by definition of $\| \cdot \|_{w^2}$. 
Further note that $R_n^*(x)^2$ equals
\begin{align}
\E&\left[ \Big\{ \widetilde{m}(x) - m(x) \Big\}^2 \right]=  \E\left( \left[ \sum_{i=1}^n w_i(x;X^n) \Big\{ f(Z_i) - m(X_i) \Big\} + \sum_{i=1}^n w_i(x;X^n) m(X_i) -  m(x) \right]^2 \right) \nonumber \\
 %&= \E\left( \left[ \sum_{i=1}^n w_i(x;X^n) \Big\{ f(Z_i) - m(X_i) \Big\}  \right]^2 \right) + \E\left[ \left\{ \sum_{i=1}^n w_i(x;X^n) m(X_i) -  m(x) \right\}^2 \right] \nonumber \\
 &\hspace{.5in} = \E \left\{ \sum_{i=1}^n w_i(x;X^n)^2 \ \sigma(X_i)^2 \right\} + \E\left[ \left\{ \sum_{i=1}^n w_i(x;X^n) m(X_i) -  m(x) \right\}^2 \right] \nonumber  \\
 &\hspace{.5in}  \geq \E\left\{  \| \sigma \|_{w^2}^2 \sum_{i=1}^n w_i(x;X^n)^2 \right\} \label{eq:lsep2}
 \end{align}
where the second line  follows from iterated expectation and independence of the samples, and the third by definition of $\| \cdot \|_{w^2}$ (and since the squared bias term from the previous line is non-negative). Therefore
\begin{align*}
\Pb \left\{   \frac{ \|\sigma\|_{w^2} |T_n | }{\| \widehat{f}-f \|_{w^2} R_n^*(x) } \geq t  \right\} &= \E\left[ \Pb \left\{ \frac{ \|\sigma\|_{w^2} |T_n | }{\| \widehat{f}-f \|_{w^2} R_n^*(x) } \geq t \Bigm| D^n, X^n \right\} \right] \\
&\leq \left( \frac{1}{t^2 R_n^*(x)^2} \right) \E\left\{  \| \sigma\|_{w^2}^2  \E\left( \frac{ T_n^2}{\| \widehat{f}-f \|_{w^2}^2  } \Bigm| D^n, X^n  \right) \right\} \\
&\leq  \left( \frac{1}{t^2 R_n^*(x)^2 } \right)  \E\left\{  \| \sigma \|_{w^2}^2 \sum_{i=1}^n w_i(x;X^n)^2 \right\}  \leq \frac{1}{ t^2}
\end{align*}
where the second line follows by Markov's inequality, the third from the bound in \eqref{eq:lsep1} and iterated expectation, and the last from the bound in \eqref{eq:lsep2}. Therefore the result follows since we can always pick $t^2=1/\epsilon$ to ensure the above probability is no more than any $\epsilon$. \\

\subsection{Proof of Proposition \ref{prop:oracle}}

Stability and consistency together imply $$ \widehat{m}(x) - \widetilde{m}(x) = \widehat\E_n\{ \widehat{b}(X) \mid X=x\} + o_\Pb\left(R_n^*(x) \right)$$  by definition. Therefore if $\widehat\E_n\{ \widehat{b}(X) \mid X=x\}=o_\Pb\left(R_n^*(x) \right)$ the result follows. \\

\subsection{Proof of Proposition \ref{prop:smoothbias}}

We have
\begin{align*}
\widehat\E_n\{ \widehat{b}(X) \mid X=x\} &= \sum_{i=1}^n w_i(x;X^n) \widehat{b}_1(X_i) \widehat{b}_2(X_i) \\
&\leq \left\{ \sum_{i=1}^n {|w_i(x;X^n)|} | \widehat{b}_1(X_i)|^p \right\}^{1/p} \left\{ \sum_{i=1}^n {|w_i(x;X^n)|} | \widehat{b}_2(X_i)|^q \right\}^{1/q}  \\
&= \left( \sum_{i=1}^n | w_i(x;X^n) | \right) \| \widehat{b}_1 \|_{w,p} \| \widehat{b}_2 \|_{w,q}
\end{align*}
where the second line follows by Holder's inequality, and the last by definition of the norms.  \\

\subsection{Proof of Theorem \ref{thm:drlearner}}

This follows from Proposition \ref{prop:oracle}, letting $f(Z)=\varphi(Z)$ and noting that 
\begin{align*}
\widehat{b}(x) &= \left\{ \frac{\pi(x)}{\widehat\pi(x)} - 1 \right\} \Big\{ \mu_1(x) - \widehat\mu_1(x) \Big\}  - \left\{ \frac{1-\pi(x)}{1-\widehat\pi(x)} - 1 \right\} \Big\{ \mu_0(x) - \widehat\mu_0(x) \Big\}  \\
&=\sum_{j=0}^1 \left\{ \frac{\widehat\pi(x) - \pi(x) }{j \widehat\pi(x) + (1-j) (1-\widehat\pi(x))} \right\}   \Big\{ \widehat\mu_j(x) - \mu_j(x) \Big\} \equiv \widehat{b}_0(x) + \widehat{b}_1(x)
\end{align*}
by iterated expectation. \\

\subsection{Proof of Corollary \ref{cor:drlsmooth}}

The rate result follows since
\begin{align*}
 \widehat{m}(x) - m(x) &=  \widehat{m}(x) - \widetilde{m}(x) + \widetilde{m}(x) - m(x) \\
 &= O_\Pb\left( n^{-\left(\frac{1}{2+d/\alpha} + \frac{1}{2 + d/\beta}\right)} \right)  + \widetilde{m}(x) - m(x) \\
 &= O_\Pb\left( n^{-\left(\frac{1}{2+d/\alpha} + \frac{1}{2 + d/\beta}\right)} \right) + O_\Pb\left( n^{\frac{-1}{2+d/\gamma}} \right)
\end{align*}
where the second equality follows from Proposition \ref{prop:smoothbias} together with Assumptions 1--3, and the third since $\widehat\E_n$ is minimax optimal and the CATE is $\gamma$-smooth. For the oracle efficiency condition, note that 
\begin{align*}
&n^{-\left(\frac{1}{2+d/\alpha} + \frac{1}{2 + d/\beta}\right)} \leq n^{-\frac{1}{2+d/\gamma}} \\
&\iff \frac{1}{2+d/\alpha} + \frac{1}{2 + d/\beta} \geq \frac{1}{2+d/\gamma} \\
&\iff 4 + d/\alpha + d/\beta \geq \frac{4 + 2d/\alpha + 2d/\beta + d^2 /\alpha\beta}{2+d/\gamma} \\
&\iff 4 + (4 + d/\alpha + d/\beta)d/\gamma \geq d^2/\alpha\beta \\
&\iff \alpha\beta \geq \frac{d^2}{4 + (4 + d/\alpha + d/\beta)d/\gamma } = \frac{d^2/4}{1 + (1 + d/4\alpha + d/4\beta)d/\gamma} 
\end{align*}
which yields the result. \\

\subsection{Proof of Theorem \ref{thm:lprlearner}}

To simplify notation, in this subsection we largely omit function arguments, for example writing $\widehat\tau=\widehat\tau(x)$, $b=b(X-x)$, $K_{hx}=K_{hx}(X)$, etc. We also define $b_0 = b(0)$ and 
\begin{align*}
\widehat\phi &= \widehat\phi(Z) = (A-\widehat\pi_a) (Y-\widehat\eta)  & 
\phi &= \phi(Z)=(A-\pi)  (Y-\eta)  \\
\widehat\varphi_a &= \widehat\varphi_a(Z) = (A-\widehat\pi_a) (A-\widehat\pi_b) &
\nu &= \nu(X)= \pi(1-\pi) . 
\end{align*}
We let $X^n$ denote all the covariates across the training and test samples $(D_{1a}^n, D_{1b}^n, D_2^n)$, and we let $D_1^n=(D_{1a}^n,D_{1b}^n)$ denote the training data. \\

First note that by definition we have $\widehat\tau = b_0^\T  \Pn(b K_{hx} \widehat\varphi_a b^\T)^{-1} \Pn(b K_{hx} \widehat\phi) $. 
Thus we begin with the central decomposition
\begin{align}
\widehat\tau - \tau &= b_0^\T \Pn(b K_{hx} \nu b^\T)^{-1} \Pn(b K_{hx} \phi) - \tau  \nonumber \\
& \hspace{.5in} + b_0^\T  \Pn(b K_{hx} \nu b^\T)^{-1} \Pn\{b K_{hx} (\widehat\phi - \phi) \} \nonumber \\
& \hspace{.5in} + b_0^\T \left\{ \Pn(b K_{hx} \widehat\varphi_a b^\T)^{-1} - \Pn(b K_{hx} \nu b^\T)^{-1}  \right\} \Pn(b K_{hx} \widehat\phi) \nonumber \\
&\equiv b_0^\T \widetilde{Q}_{hx}^{-1} \Pn(b K_{hx} \phi) - \tau  \label{eq:lpr1} \\
& \hspace{.5in} + b_0^\T  \widetilde{Q}_{hx}^{-1} \Pn\{b K_{hx} (\widehat\phi - \phi) \} \label{eq:lpr2} \\
& \hspace{.5in} + b_0^\T \left( \widehat{Q}_{hx}^{-1} - \widetilde{Q}_{hx}^{-1}  \right) \Pn(b K_{hx} \widehat\phi) \label{eq:lpr3}
\end{align}
where we define
\begin{align}
\widetilde{Q}_{hx} &\equiv \Pn\Big\{ b(X-x) K_{hx}(X) \nu(X) b(X-x)^\T\Big\} \label{eq:qtildehat} \\
\widehat{Q}_{hx} &\equiv \Pn\Big\{ b(X-x) K_{hx}(X) \widehat\varphi_a(Z) b(X-x)^\T\Big\} . \nonumber
\end{align}

\bigskip

The general approach in this proof is to use conditional error bounds for each term in the decomposition above, from which one arrives at bounds in probability (cf.\ Lemma 6.1 of \citet{chernozhukov2018double}). \\

The term on the right-hand side of \eqref{eq:lpr1} is the difference between $\tau$ and an oracle version of the lp-R-Learner that has access to the true nuisance functions and so is built from $\nu$ and $\phi$; we will show that it attains the same order as the oracle rate $n^{-\gamma/(2\gamma+d)}$. The term \eqref{eq:lpr2} captures the error from estimating $(\widehat\pi_b,\widehat\eta)$ in $\widehat\phi$; we will show its order is the product of the biases from estimating $(\widehat\pi_b,\widehat\eta)$ plus a typically smaller variance term. The term \eqref{eq:lpr3} captures the error from estimating $(\widehat\pi_a,\widehat\pi_b)$ in $\widehat\phi_a$; we will show it behaves similarly to \eqref{eq:lpr2}, except involving the product of the biases of $(\widehat\pi_a,\widehat\pi_b)$. In regimes where the oracle rate is not achievable and the propensity score $\pi$ is smoother than the regression function $\mu$, we will show that the term \eqref{eq:lpr2} dominates. \\

\subsubsection{Term \eqref{eq:lpr1}}

The analysis of the term in \eqref{eq:lpr1} follows that of a standard local polynomial estimator of pseudo-outcome $\phi/\nu$ on $X$ with a special choice of kernel. This is because we can write 
\begin{align*}
b_0^\T \Pn(b K_{hx} \nu b^\T)^{-1} \Pn(b K_{hx} \phi) &= \sum_{i=1}^n w_i(x;X^n) \frac{\phi(Z_i)}{\nu(X_i)}
\end{align*}
where the weights are
\begin{align}
w_i(x;X^n) &\equiv \frac{1}{n} b(0)^\T \widetilde{Q}_{hx}^{-1} b(X_i-x) K_{hx}(X_i) \nu(X_i) .  \label{eq:step2wts} 
\end{align}
Thus the oracle estimator in \eqref{eq:lpr1} is a local polynomial estimator of the regression of $\phi/\nu$ on $X$ using scaled kernel function $K_{hx} \nu$, which has the same support as $K_{hx}$ and has a smaller upper bound since $\nu \leq 1/4$. \\

The above implies that the weights $w_i$ satisfy analogues of Proposition 1.12 and Lemma 1.3 in \citet{tsybakov2009introduction}, i.e., that they reproduce polynomials in $X$ up to degree $\lfloor\gamma\rfloor$, and have the localizing properties given in the following lemma. \\

\begin{lemma} \label{lem:lprwts}
Assume (i) $|K(u) | \leq K_{max} \one(\|u\| \leq 1)$ and (ii) $\sup_x \| b(x) \| \leq B$, and define
$$ \lambda_n \equiv \left\| \widetilde{Q}_{hx}^{-1} \right\| \ \text{ and } \ \xi_n \equiv \Pn\Big(\|X - x \| \leq h \Big) / h^d . $$
Then the weights in \eqref{eq:step2wts} satisfy:
\begin{enumerate}
\item $\max_{i} | w_i(x;X^n)| \lesssim \lambda_n / nh^d $. 
\item $\sum_{i=1}^n |w_i(x;X^n| \lesssim  \lambda_n \xi_n  $.
\item $w_i(x;X^n) = 0$ when $\|X_i - x \| > h$.
\end{enumerate}
\end{lemma}

\begin{proof}
Property (3) follows from Assumption (i). 
Properties (1) and (2) also follow immediately after noting that
\begin{align*}
|w_i(x;X^n)| \leq \frac{K_{max} B}{4nh^d} \left\| \widetilde{Q}_{hx}^{-1} \right\|  \one\Big(\|X_i - x \| \leq h \Big)
\end{align*}
using the submultiplicative property of the operator norm, with Assumptions (i)--(ii) and the facts that $\|b(0)\|=1$ and $\nu =\pi(1-\pi) \leq 1/4$. 
\end{proof}

\bigskip

Note also that $\E(\phi/\nu \mid X)=\tau$ since
\begin{align*}
\E\{\phi(Z)  \mid X\} &= \E\Big[ \{A-\pi(X) \}\{Y - \eta(X) \}  \Bigm| X \Big] = \pi(X) \{ 1- \pi(X)\} \tau(X)
\end{align*}
by iterated expectation.  Therefore by the same logic as in Proposition 1.13 of \citet{tsybakov2009introduction}, using Lemma \ref{lem:lprwts} with the H\"{o}lder condition on $\tau$, we have
\begin{align*} 
\E\Big\{ b_0^\T &\Pn(b K_{hx} \nu b^\T)^{-1} \Pn(b K_{hx} \phi) - \tau  \Bigm| X^n \Big\} = \sum_{i=1}^n w_i(x;X^n) \Big\{ \tau(X_i) - \tau(x) \Big\} \\
&= \sum_{i=1}^n w_i(x;X^n) \left[ \Big\{ \tau_{\gamma,x}(X_i) - \tau(x) \Big\} + \Big\{ \tau_{\gamma,x}(X_i) - \tau(X_i)  \Big\} \right] \\
&\lesssim  \sum_{i=1}^n | w_i(x;X^n) | \| X_i - x \|^\gamma \one(\| X_i - x \| \leq h) \lesssim h^\gamma  \lambda_n  \xi_n
\end{align*}
where in the second line $\tau_{\gamma,x}(X_i)=\sum_{|j| \leq \lfloor\gamma\rfloor} \frac{(X_i - x)^j}{j!} D^j\tau(x)$ is the $\lfloor\gamma\rfloor$-order multivariate Taylor approximation of $\tau$ at $x$ evaluated at $X_i$, the third line follows by the polynomial-reproducing property of $w_i$, the H\"{o}lder assumption on $\tau$,  and Property 3 of Lemma \ref{lem:lprwts}, and the last line by Property 2 of Lemma \ref{lem:lprwts}. \\

For the variance we have
\begin{align*} 
\var\Big\{ b_0^\T &\Pn(b K_{hx} \nu b^\T)^{-1}  \Pn(b K_{hx} \phi)  \Bigm| X^n \Big\} = \sum_{i=1}^n w_i(x;X^n)^2 \var(\phi \mid X=X_i) \\
& \lesssim \max_{i} |w_i(x;X^n)| \sum_{i=1}^n |w_i(x;X^n)| \lesssim \frac{\lambda_n^2 \xi_n}{nh^d} 
\end{align*}
since the variance of $\phi$ is bounded, and using Properties 1--2 of Lemma \ref{lem:lprwts}. Therefore term \eqref{eq:lpr1} satisfies
\begin{align}
\E&\left[ \Big\{b_0^\T \widetilde{Q}_{hx}^{-1} \Pn(b K_{hx} \phi) - \tau \Big\}^2 \Bigm| X^n \right] \lesssim \ \left( h^{2\gamma}  \xi_n + \frac{1}{nh^d} \right) \lambda_n^2 \xi_n.  \label{eq:lpr1rate_fs}
\end{align}
Since $\xi_n= O_\Pb(1)$ and $\lambda_n=O_\Pb(1)$ by assumption, the above rate will end up matching the classical $n^{-\gamma/(2\gamma+d}$ rate, when balancing bias and variance by taking $h \sim n^{-1/(2\gamma+d)}$. 

\bigskip

\subsubsection{Term \eqref{eq:lpr2}}

Now we bound the conditional mean and variance of term \eqref{eq:lpr2}.  First note we have
\begin{align}
\E(\widehat\phi - \phi \mid D^n,X^n) &= \E\Big\{ (A - \widehat\pi_a) (Y - \widehat\eta) - (A-\pi) (Y - \eta) \mid D^n,X^n \Big\} \nonumber \\
&=  \E\Big\{ (A - \pi + \pi - \widehat\pi_a) (Y - \eta + \eta - \widehat\eta) - (A-\pi) (Y - \eta) \mid D^n,X^n \Big\} \nonumber \\
&= \{\widehat\pi_a(X_i)-\pi(X_i) \} \{\widehat\eta(X_i) - \eta(X_i)\} \equiv \widehat{R}_2(X_i) \label{eq:rlr2}
\end{align}
by iterated expectation. Let
\begin{align*}
\mathcal{B}_n(x;\widehat{f}) = \E\{\widehat{f}(x) \mid X^n\} - f(x) 
\end{align*}
denote the pointwise conditional bias of a generic estimator $\widehat{f}$ of $f$ at $x$.  

Then the conditional mean of term \eqref{eq:lpr2} is
\begin{align}
\E\Big[ b_0^\T  \Pn(b K_{hx} \nu b^\T)^{-1} &\Pn\{b K_{hx} (\widehat\phi - \phi) \} \mid X^n \Big] =  \sum_{i=1}^n w_i(x;X^n) \nu(X_i)^{-1}  \E(\widehat\phi_i - \phi_i \mid X^n)  \nonumber \\
&= \sum_{i=1}^n w_i(x;X^n) \nu(X_i)^{-1} \E\{\widehat{R}_2(X_i) \mid X^n\}  \nonumber \\
&= \sum_{i=1}^n w_i(x;X^n) \nu(X_i)^{-1} \ \mathcal{B}_n(X_i;\widehat\pi_a)  \mathcal{B}_n(X_i;\widehat\eta)  \nonumber \\
&\lesssim \sum_{i=1}^n \frac{\lambda_n}{nh^d} \left| \mathcal{B}_n(X_i;\widehat\pi_a)  \mathcal{B}_n(X_i;\widehat\eta) \right| \one(\| X_i - x \| \leq h) \nonumber \\
& \lesssim \sup_{\|x'-x|| \leq h} \left| \mathcal{B}_n(x';\widehat\pi_a) \right| \left| \mathcal{B}_n(x';\widehat\eta) \right| \lambda_n \xi_n  \lesssim \ k^{-\alpha/d} k^{-\beta/d} \Big( \lambda_n \xi_n  \Big) \label{eq:lpr2mean}
\end{align}
where the second line follows by iterated expectation, the third since $\widehat\pi_a \ind \widehat\eta$,  the fourth and fifth by Properties 1--3 of Lemma \ref{lem:lprwts} and since $\nu \geq \epsilon(1-\epsilon)$, and the last by Condition \ref{eq:nuisbias} via Assumption \ref{ass:nuiserr}. For the conditional variance we have the decomposition
\begin{align}
\var&\Big[ b_0^\T  \Pn(b K_{hx}  \nu b^\T)^{-1} \Pn\{b K_{hx} (\widehat\phi - \phi) \} \mid X^n \Big] = \var\left[ \sum_{i=1}^n w_i(x;X^n) \left\{ \frac{\widehat\phi(Z_i) - \phi(Z_i)}{\nu(X_i)} \right\} \Bigm| X^n \right] \nonumber \\
&=  \E\left(\var\left[ \sum_{i=1}^n w_i(x;X^n) \left\{ \frac{\widehat\phi(Z_i) - \phi(Z_i)}{\nu(X_i)} \right\} \Bigm| D^n, X^n \right] \Bigm| X^n \right) \label{eq:lpr2v1} \\
& \hspace{.5in} +  \var\left(\E\left[ \sum_{i=1}^n w_i(x;X^n) \left\{ \frac{\widehat\phi(Z_i) - \phi(Z_i)}{\nu(X_i)} \right\} \Bigm| D^n, X^n \right] \Bigm| X^n \right) \label{eq:lpr2v2}
\end{align}
For the term in  \eqref{eq:lpr2v1} note that
\begin{align*}
\var\bigg[ \sum_{i=1}^n \bigg\{ \frac{w_i(x;X^n) }{\nu(X_i)} \bigg\} & \left(\widehat\phi_i - \phi_i \right) \Bigm| D^n, X^n \bigg] =  \sum_{i=1}^n \left\{ \frac{w_i(x;X^n)}{v(X_i)} \right\}^2 \var\left(\widehat\phi_i - \phi_i \mid D^n, X^n \right)  \\
& \lesssim \frac{1}{\epsilon(1-\epsilon)}  \max_{i} | w_i(x;X^n) | \sum_{i=1}^n | w_i(x;X^n)|   \lesssim \frac{\lambda_n^2 \xi_n}{nh^d} 
\end{align*}
where the second line used that $\var(\widehat\phi - \phi \mid D^n,X^n)$ and $1/\nu$ are bounded, and the last Properties 1--2 of Lemma \ref{lem:lprwts}. \\

The second term \eqref{eq:lpr2v2} equals
\begin{align}
\var&\left[ \sum_{i=1}^n \left\{ \frac{w_i(x; X^n)}{\nu(X_i)} \right\} \widehat{R}_2(X_i) \Bigm| X^n \right] = \sum_{i=1}^n  \left\{ \frac{w_i(x; X^n)}{\nu(X_i)} \right\}^2 \var\left\{ \widehat{R}_2(X_i) \mid X^n \right\} \label{eq:catevar1} \\
& \hspace{.5in} + \sum_{i \neq j} \left\{ \frac{w_i(x; X^n)w_j(x; X^n)}{\nu(X_i) \nu(X_j)} \right\} \cov\left\{ \widehat{R}_2(X_i) , \widehat{R}_2(X_j) \mid X^n \right\} \label{eq:catevar2}
\end{align}
For the first term \eqref{eq:catevar1} above we have
\begin{align}
\var&\left\{ \widehat{R}_2(X_i) \mid X^n \right\} = \var\Big[ \{ \widehat\pi(X_i) - \pi(X_i)\} \{ \widehat\eta(X_i) - \eta(X_i)\} \mid X^n \Big] \nonumber \\
&= \var\{ \widehat\pi(X_i) \mid X^n\} \var\{ \widehat\eta(X_i) \mid X^n\} +  \var\{ \widehat\pi(X_i) \mid X^n\} \Big[ \E\{ \widehat\eta(X_i) - \eta(X_i) \mid X^n\}\Big]^2  \nonumber\\
& \hspace{.5in} + \var\{ \widehat\eta(X_i) \mid X^n\} \Big[ \E\{ \widehat\pi(X_i) - \pi(X_i) \mid X^n\}\Big]^2 \nonumber \\
&\lesssim \left( \frac{k}{n} \right)^2 + \frac{k}{n} \left( k^{-2\alpha/d} + k^{-2\beta/d} \right)
\end{align}
where the second line uses the fact that $\var(V_1 V_2) = \var(V_1) \var(V_2) + \var(V_1) \E(V_2)^2 + \var(V_2) \E(V_1)^2$ if $V_1 \ind V_2$, and the third the bias and variance from Conditions 1b-1c. (A slightly worse  bound could have used $\var(V_1V_2) \leq \E(V_1^2) \E(V_2^2)$ in the second line, which would give the expression in the third line plus a $k^{-2(\alpha+\beta)/d}$ term, which is smaller order in the undersmoothing regime.) Therefore when $\max(k^{-2\alpha/d}, k^{-2\beta/d}) = k^{-2(\alpha \wedge \beta)/d} \lesssim k/n$, the first term in  \eqref{eq:catevar1} satisfies
$$ \sum_{i=1}^n  \left\{ \frac{w_i(x; X^n)}{\nu(X_i)} \right\}^2 \var\left\{ \widehat{R}_2(X_i) \mid X^n \right\} \lesssim \frac{\lambda_n^2 \xi_n}{nh^d} \left( \frac{k}{n} \right)^2 $$
using Properties 1--2 of Lemma \ref{lem:lprwts}. \\

For the second term in  \eqref{eq:catevar2} we have
\begin{align}
\cov&\left\{ \widehat{R}_2(X_i) , \widehat{R}_2(X_j) \mid X^n \right\} \nonumber \\
& = \cov\{ \widehat\pi(X_i) , \widehat\pi(X_j) \mid X^n\} \E\Big[ \{ \widehat\eta(X_i)-\eta(X_i)\}\{\widehat\eta(X_j) - \eta(X_j) \} \mid X^n \Big] \nonumber \\
& \hspace{.5in} +  \cov\{ \widehat\eta(X_i) , \widehat\eta(X_j) \mid X^n\} \E\Big[ \{ \widehat\pi(X_i)-\eta(X_i)\}\{\widehat\pi(X_j) - \eta(X_j) \} \mid X^n \Big] \nonumber \\
& \hspace{.5in} + \cov\{ \widehat\pi(X_i) , \widehat\pi(X_j) \mid X^n\} \cov\{ \widehat\eta(X_i) , \widehat\eta(X_j) \mid X^n\} \nonumber \\
&\lesssim \Big| \cov\{ \widehat\pi(X_i) , \widehat\pi(X_j) \mid X^n\} \Big| \left( k^{-2\beta/d} + \frac{k}{n} \right) + \Big| \cov\{ \widehat\eta(X_i) , \widehat\eta(X_j) \mid X^n\} \Big| \left( k^{-2\alpha/d} + \frac{k}{n} \right) \nonumber \\
&\hspace{.5in} + \Big| \cov\{ \widehat\pi(X_i) , \widehat\pi(X_j) \mid X^n\} \cov\{ \widehat\eta(X_i) , \widehat\eta(X_j) \mid X^n\} \Big| \label{eq:covnuis1}
\end{align}
where the second line uses the fact that $\cov(A_i B_i, A_j B_j) = \cov(A_i,A_j) \E(B_i B_j) + \cov(B_i,B_j) \E(A_i A_j) + \cov(A_i,A_j) \cov(B_i,B_j)$ whenever $(A_i,A_j) \ind (B_i,B_j)$, and the third Cauchy-Schwarz. Now $\cov\{ \widehat\pi(X_i) , \widehat\pi(X_j) \mid X^n\} $ equals
\begin{align}
& \cov\left\{ \sum_{\ell=1}^n w_{\ell\alpha}(X_i; X^n) A_\ell ,  \sum_{\ell=1}^nw_{\ell\alpha}(X_j; X^n) A_\ell  \mid X^n \right\} \one\left(\|X_i - X_j \| \leq 2k^{-1/d}\right) \nonumber \\
& \leq \sqrt{ \var\{ \widehat\pi(X_i) \mid X^n\} }\sqrt{ \var\{ \widehat\pi(X_j) \mid X^n\}} \one\left(\|X_i - X_j \| \leq 2k^{-1/d}\right) \nonumber\\
&\lesssim \left( \frac{k}{n} \right) \one\left(\|X_i - X_j \| \leq 2k^{-1/d}\right) \label{eq:covnuis2}
\end{align}
where the first line uses Property 3 of Lemma \ref{lem:lprwts}, which implies $\widehat\pi(X_i)$ and $\widehat\pi(X_j)$ are built from different observations if $X_i$ and $X_j$ are far enough apart, the second Cauchy-Schwarz, and the 
third Condition \ref{eq:nuiswts}. 
Therefore when $k^{-2(\alpha \wedge \beta)/d} \lesssim k/n$, and defining
\begin{align*}
\omega_{ni} & \equiv \frac{k}{n-1} \sum_{j \neq i} \one\left(\|X_i - X_j \| \leq 2k^{-1/d}\right)  \\
\xi'_n &\equiv \frac{1}{nh^d} \sum_{i=1}^n \one(\|X_i-x\| \leq h) \omega_{ni}
\end{align*}
the second term in \eqref{eq:catevar2} satisifies
\begin{align*}
 \sum_{i \neq j} & \left\{ \frac{w_i(x; X^n)w_j(x; X^n)}{\nu(X_i) \nu(X_j)} \right\} \cov\left\{ \widehat{R}_2(X_i) , \widehat{R}_2(X_j) \mid X^n \right\} \\
 &\lesssim \left( \frac{k}{n}\right)^2  \sum_{i \neq j} \left\{ \frac{ |w_i(x; X^n)w_j(x; X^n)|}{\nu(X_i) \nu(X_j)} \right\} \one\left(\|X_i - X_j \| \leq 2k^{-1/d}\right) \\
 &\lesssim \left( \frac{\lambda_n}{nh^d}\right)^2  \left( \frac{k}{n}\right)^2 \sum_{i \neq j} \one(\|X_i-x\| \leq h) \one\left(\|X_i - X_j \| \leq 2k^{-1/d}\right)  =  \xi_n' \left( \frac{\lambda_n^2}{nh^d}\right)  \left( \frac{k}{n}\right) 
\end{align*}
where the first inequality follows from \eqref{eq:covnuis1} and \eqref{eq:covnuis2}, the second by Properties 1 and 3 of  Lemma \ref{lem:lprwts}, and the last by definition. \\

Therefore combining bounds on the four terms \eqref{eq:lpr2mean},  \eqref{eq:lpr2v1},  \eqref{eq:catevar1}, and \eqref{eq:catevar2}, we have that the term \eqref{eq:lpr2} satisfies
\begin{align}
\E\left( \Big[ b_0^\T  \widetilde{Q}_{hx}^{-1} \Pn\{b K_{hx} (\widehat\phi - \phi) \}  \Big]^2 \Bigm| X^n \right)  &\lesssim k^{-4s/d} (\lambda_n \xi_n)^2  + \frac{\lambda_n^2}{nh^d} \left(  \xi_n +  \left( \frac{k}{n} \right)  \left( \xi'_n  + \xi_n \frac{ k}{n} \right) \right)
 \label{eq:lpr2rate_fs}
\end{align}

\bigskip

\subsubsection{Term \eqref{eq:lpr3}}

Note term \eqref{eq:lpr3} equals
\begin{align}
b_0^\T \Big\{ \Pn(b K_{hx} \widehat\varphi_a b^\T)^{-1} &- \Pn(b K_{hx} \nu b^\T)^{-1}  \Big\} \Pn(b K_{hx} \widehat\phi) = b_0^\T \left( \widehat{Q}_{hx}^{-1} - \widetilde{Q}_{hx}^{-1} \right) \Pn(bK_{hx} \widehat\phi) \nonumber \\
&= \left\{ b_0^\T \widetilde{Q}_{hx}^{-1} \left( \widetilde{Q}_{hx} - \widehat{Q}_{hx} \right) \right\}  \left\{ \widehat{Q}_{hx}^{-1}   \Pn(bK_{hx} \widehat\phi)  \right\} \label{eq:lpr3decomp}
\end{align}
For the first term in the product in \eqref{eq:lpr3decomp} we have
\begin{align*}
b_0^\T \widetilde{Q}_{hx}^{-1} \left( \widehat{Q}_{hx} - \widetilde{Q}_{hx} \right) &= \sum_{i=1}^n \left\{ \frac{w_i(x;X^n)}{\nu(X_i)} \right\}  \Big\{ \widehat\varphi_a(Z_i)-\nu(X_i) \Big\} b(X_i-x)^\T 
\end{align*}
which is a finite-dimensional vector with elements
\begin{align*}
 \sum_{i=1}^n  \left\{ \frac{b_\ell(X_i-x)}{\nu(X_i)} \right\} w_i(x;X^n) \Big\{ \widehat\varphi_a(Z_i)-\nu(X_i) \Big\} 
\end{align*}
which we can tackle with similar logic as for term \eqref{eq:lpr2}. First note
\begin{align}
\E( \widehat\varphi_a - \nu \mid D^n,X^n) &= \E\Big\{ (A - \widehat\pi_a) (A - \widehat\pi_b)  - (A-\pi)^2 \mid D^n,X^n \Big\} \nonumber \\
&=  \E\Big\{ (A - \pi + \pi - \widehat\pi_a) (A - \pi + \pi - \widehat\pi_b) - (A-\pi)^2\mid D^n,X^n \Big\} \nonumber \\
&= \{\widehat\pi_a(X_i)-\pi(X_i) \} \{\widehat\pi_b(X_i)-\pi(X_i) \} \equiv \widehat{R}_{2\pi}(X_i) \label{eq:rlr2pi}
\end{align}
by iterated expectation. Defining $\mathcal{B}_n(x;\widehat{f})$ as in the previous subsection,  the conditional means of the elements of  \eqref{eq:lpr3decomp} (omitting subscripts $\ell$) are
\begin{align*}
\E\left\{ b_0^\T \widetilde{Q}_{hx}^{-1} \left( \widehat{Q}_{hx} - \widetilde{Q}_{hx} \right) \Bigm| X^n \right\}_\ell &=  \sum_{i=1}^n  \left\{ \frac{b(X_i-x)}{\nu(X_i)} \right\} w_i(x;X^n) \E\Big\{ \widehat\varphi_a(Z_i)-\nu(X_i) \mid X^n \Big\}  \\
&= \sum_{i=1}^n \left\{ \frac{b(X_i-x)}{\nu(X_i)} \right\}w_i(x;X^n)  \E\{\widehat{R}_{2\pi}(X_i) \mid X^n\}  \\
&= \sum_{i=1}^n \left\{ \frac{b(X_i-x)}{\nu(X_i)} \right\} w_i(x;X^n)  \ \mathcal{B}_n(X_i;\widehat\pi_a)  \mathcal{B}_n(X_i;\widehat\pi_b)  \\
&\lesssim \sum_{i=1}^n \frac{\lambda_n}{nh^d} \left| \mathcal{B}_n(X_i;\widehat\pi_a)  \mathcal{B}_n(X_i;\widehat\pi_b) \right| \one(\| X_i - x \| \leq h) \\
& \lesssim \sup_{\|x'-x|| \leq h} \left| \mathcal{B}_n(x';\widehat\pi_a) \right| \left| \mathcal{B}_n(x';\widehat\pi_b) \right|  (\lambda_n \xi_n)  \lesssim \ k^{-2\alpha/d} \Big( \lambda_n \xi_n \Big)
\end{align*}
where the second line follows by iterated expectation, the third since $\widehat\pi_a \ind \widehat\pi_b$,  the fourth and fifth by Properties 1--3 of Lemma \ref{lem:lprwts} and since $\nu \geq \epsilon(1-\epsilon)$, and the last by Condition \ref{eq:nuisbias} via Assumption \ref{ass:nuiserr}. \\

The analysis of the conditional variance follows exactly the same logic as for term \eqref{eq:lpr2}, and is of the same order. For the second term in the product in \eqref{eq:lpr3decomp} we have
\begin{align*}
\left|\left|  \widehat{Q}_{hx}^{-1}   \Pn(bK_{hx} \widehat\phi)  \right|\right| \lesssim \widehat\lambda_n \left|\left|  \Pn(bK_{hx} \widehat\phi)  \right|\right|
\end{align*}
and note
\begin{align*}
\E\{\| \Pn(b K_{hx} \widehat\phi)\|^2 \mid X^n \} &= \sum_j \E\{\Pn(b_j K_{hx} \widehat\phi)^2 \mid X^n\} \\
&= \sum_j \left[ \E\{ \Pn(b_j K_{hx} \widehat\phi) \mid X^n\}^2+ \var\{ \Pn(b_j K_{hx} \widehat\phi) \mid X^n\} \right] .
\end{align*}
%we have that given $X^n$ the term $\| \Pn(b K_{hx} \widehat\phi) \|$ is bounded in probability as the sum of the mean and standard deviation of its finitely many components, by Markov's inequality.  \\

Therefore for $b_\ell$ the $\ell$-th component of $b(X_i-x)$
\begin{align*}
\E\Big\{ \Pn(b_\ell K_{hx} \widehat\phi) \mid X^n \Big\} &= \Pn\Big\{ b_\ell K_{hx} \E(\widehat\phi_i \mid X^n) \Big\} \lesssim \frac{1}{nh^d} \sum_{i=1}^n \one(\| X_i - x \| \leq h ) = \xi_n
\end{align*}
using boundedness of the kernel $K$ and observations $Z$. For the variance we have
\begin{align}
\var\Big\{ \Pn(b_\ell K_{hx} \widehat\phi) \mid X^n \Big\} &= \E\left[ \var\Big\{ \Pn(b_\ell K_{hx} \widehat\phi) \mid D^n, X^n \Big\} \Bigm| X^n \right] \label{eq:last1} \\
& \hspace{.5in} +  \var\left[ \E\Big\{ \Pn(b_\ell K_{hx} \widehat\phi) \mid D^n, X^n \Big\} \Bigm| X^n \right] \label{eq:last2}
\end{align}
For the term in \eqref{eq:last1} note
\begin{align*}
 \var\Big\{ \Pn(b_\ell K_{hx} \widehat\phi) \mid D^n, X^n \Big\}  &=  \left( \frac{1}{nh^d} \right)^2 \sum_{i=1}^n b_\ell(X_i-x)^2 K\left(\frac{X_i-x}{h}\right)^2 \var(\widehat\phi_i \mid D^n,X^n) \\
&\lesssim \left( \frac{1}{nh^d} \right)^2 \sum_{i=1}^n \one(\| X_i - x \| \leq h) = \frac{\xi_n}{nh^d}
\end{align*}
using the boundedness of the kernel, $X$, and $\widehat\phi$. For the term in \eqref{eq:last2} 
\begin{align*}
\var&\left[ \E\Big\{ \Pn(b_\ell K_{hx} \widehat\phi) \mid D^n, X^n \Big\} \Bigm| X^n \right] = \var\left[ \Pn\Big\{ b_\ell K_{hx}\E( \widehat\phi \mid D^n, X^n) \Big\} \Bigm| X^n \right] \\
&= \var\left\{ \Pn\Big\{ b_\ell K_{hx}  \widehat{R}_2(X_i)  \Bigm| X^n \right\} \lesssim \frac{ 1}{nh^d} \left(  \xi_n +  \left( \frac{k}{n} \right)  \left( \xi'_n  + \xi_n \frac{ k}{n} \right) \right)
\end{align*}
where the second line follows from the definition of $\widehat{R}_2(X_i)$ and since $\E(\phi \mid X_i)$ is constant given $X^n$.  The rest follows the same logic as for  the term in \eqref{eq:lpr2v2}, noting that $\lambda_n$ terms do not appear since here there is only a kernel weight, and so no corresponding $\| \widetilde{Q}_{hx}^{-1} \|$ term.  \\

Therefore the square of the term \eqref{eq:lpr3} has conditional expectation bounded above by a constant multiple of
\begin{align}
&\left\{ k^{-4\alpha/d} (\lambda_n \xi_n)^2 + \frac{\lambda_n^2 }{nh^d}\left(  \xi_n +  \left( \frac{k}{n} \right)  \left( \xi'_n  + \xi_n \frac{ k}{n} \right) \right) \right\}\nonumber \\
& \hspace{.5in} \times \widehat\lambda_n \left\{ \xi_n^2 + \frac{1}{nh^d}  \left(  \xi_n +  \left( \frac{k}{n} \right)  \left( \xi'_n  + \xi_n \frac{ k}{n} \right) \right) \right\}  \label{eq:lpr3rate_fs} .
\end{align}

\bigskip

\subsubsection{Combining Bounds}

Now we use the following lemma (or equivalently Lemma 6.1 of \citet{chernozhukov2018double}) to deduce unconditional convergence from the previously derived conditional bounds. \\

\begin{lemma} \label{lem:condbd}
Suppose a random variable $Z_n$ satisfies
$$ | \E(Z_n \mid X^n) | \lesssim b_n \ \text{ and } \ \var(Z_n \mid X^n) \lesssim s_n^2  $$
for some $b_n=b(X^n)$ and $s_n^2=s(X_n)^2$. 
Then 
$Z_n = O_\Pb(b_n + s_n). $
\end{lemma}

\begin{proof}
Note for some $M_\epsilon$ depending on any $\epsilon >0$ we have
\begin{align*}
\Pb\left( \frac{|Z_n|}{b_n+s_n} \geq M_\epsilon \right) &= \E\left\{ \Pb\left( \frac{|Z_n|}{b_n+s_n} \geq M_\epsilon \Bigm| X^n \right) \right\} \\
&\leq \E\left\{  \frac{ \E(Z_n^2 \mid X^n) }{(b_n+s_n)^2 M_\epsilon^2} \right\} \\
&\leq  \E\left\{ \frac{ \E(Z_n \mid X^n)^2 +\var(Z_n \mid X^n)  }{(b_n^2+s_n^2) M_\epsilon^2} \right\} \leq \frac{C}{M_\epsilon^2} 
\end{align*}
where the first equality uses iterated expectation, the second Markov's inequality, the third the fact that $(b_n+s_n)^2 \geq b_n^2 + s_n^2$, and the last the bounds on $|\E(Z_n \mid X^n)|$ and $\var(Z_n \mid X^n)$. The result follows taking $M_\epsilon=\sqrt{C/\epsilon}$. 
\end{proof}

\bigskip

Recall our conditional bounds  in \eqref{eq:lpr1rate_fs}, \eqref{eq:lpr2rate_fs}, and \eqref{eq:lpr3rate_fs} involve the quantities $(\lambda_n, \widehat\lambda_n, \xi_n, \xi_n')$. The quantities $\lambda_n$ and $\widehat\lambda_n$ are $O_\Pb(1)$ by Assumption \ref{ass:eigen}, and in the following lemmas we show that $\xi_n$ and $\xi_n'$ are as well, as long as the covariate density is bounded (Assumption \ref{ass:zestbd}). \\

\begin{lemma}
Assume $X$ has a density bounded above by some $C<\infty$. Then 
$$ \xi_n \equiv \Pn\Big( \| X-x \| \leq h \Big)/h^d = O_\Pb(1) . $$
\end{lemma}

\begin{proof}
First note that
\begin{align*}
\E(\xi_n) = \frac{1}{h^d} \int \one\Big( \| t-x \| \leq h \Big) \ p(t) \ dt \leq \frac{C}{h^d} \frac{\pi^{d/2} h^d}{\Gamma(1+d/2)} \leq 5.3 C
\end{align*}
using the bound on the density and the volume of a $d$-ball of radius $h$. Therefore for any $\epsilon > 0$ we have
\begin{align*}
\Pb( \xi_n \geq 5.3 C/\epsilon ) \leq \epsilon
\end{align*}
by Markov's inequality, which yields the result.
\end{proof}

\bigskip

\begin{lemma}
Assume $X$ has a density bounded above by some $C<\infty$.  Define
$$ \omega_{ni} \equiv  \frac{k}{n-1}  \sum_{j \neq i} \one\Big( \| X_i - X_j \| \leq 2/k^{1/d} \Big) . $$
Then
$$ \xi_n' \equiv \frac{1}{nh^d} \sum_{i=1}^n \one\Big( \| X_i-x \| \leq h \Big) \omega_{ni} = O_\Pb(1) $$
\end{lemma}

\begin{proof}
For any $i$ we have
\begin{align*}
\E(\omega_{ni} \mid X_i) &= k \int \one\Big(\|X_i-t \| \leq 2/k^{1/d} \Big) \ p(t) \ dt \leq 2^d 5.3 C
\end{align*}
just as in Lemma 3.  Therefore
\begin{align*}
\E\left\{ \frac{1}{nh^d} \sum_{i=1}^n \one\Big( \| X_i - x \| \leq h \Big) \omega_{ni} \right\} &= \E\left\{ \frac{1}{nh^d} \sum_{i=1}^n \one\Big( \| X_i - x \| \leq h \Big) \E( \omega_{ni} \mid X_i) \right\}  \\
&\leq 2^d 5.3 C \E(\xi_n) \leq 2^d (5.3 C)^2
\end{align*}
by iterated expectation and Lemma 3. Thus the result follows by Markov's inequality.  
\end{proof}

\bigskip

Therefore, combining the bounds  in \eqref{eq:lpr1rate_fs}, \eqref{eq:lpr2rate_fs}, and \eqref{eq:lpr3rate_fs}, together with the facts that $(\lambda_n, \widehat\lambda_n, \xi_n, \xi_n')$ are all $O_\Pb(1)$ due to Assumptions \ref{ass:zestbd} and \ref{ass:eigen}, we have the unconditional convergence results
\begin{align*} 
 b_0^\T \widehat{Q}_{hx}^{-1} \Pn(b K_{hx} \phi) - \tau  &= O_\Pb \left( h^\gamma + \frac{1}{\sqrt{nh^d}} \right) \\
b_0^\T  \Pn(b K_{hx} \nu b^\T)^{-1} \Pn\{b K_{hx} (\widehat\phi - \phi) \} &= O_\Pb \left( k^{-2s/d} + \frac{1}{\sqrt{nh^d}} \left( 1 + \sqrt{\frac{k}{n}} + \frac{k}{n} \right) \right)  \\
b_0^\T \widetilde{Q}_{hx}^{-1} \left( \widetilde{Q}_{hx} - \widehat{Q}_{hx} \right)   \widehat{Q}_{hx}^{-1}   \Pn(bK_{hx} \widehat\phi) &= O_\Pb  \left( k^{-2\alpha/d}  + \frac{1}{\sqrt{nh^d}} \left( 1 + \sqrt{\frac{k}{n}}+ \frac{k}{n} \right) \right)  .
\end{align*}

Note we can discard the $\sqrt{k/n}$ terms since, if $k \leq n$ then the constant term 1 dominates the variance multiplier, whereas if $k \geq n$ then $k/n \geq \sqrt{k/n}$ and the $k/n$ term dominates. 

\bigskip

\section{R Code} \label{sec:appendix}

Piecewise polynomial model from motivating example in Section \ref{sec:motiv}:

\begin{verbatim}
set.seed(1234)
expit <- function(x){ exp(x)/(1+exp(x)) }; logit <- function(x){ log(x/(1-x)) }
n <- 4*2000; nsim <- 500; rateseq <- seq(0.1,0.5,by=0.05); res2 <- NULL

for (rate in rateseq){

  res <- data.frame(matrix(nrow=nsim,ncol=4))
  colnames(res) <- c("plugin","xl","drl","oracle.drl")

for (i in 1:nsim){
  
## simulate data
s <- sort(rep(1:4,n/4)); x <- (runif(n,-1,1)); ps <- 0.1 + 0.8*(x>0)
mu0 <- (x <= -.5)*0.5*(x+2)^2 + (x/2+0.875)*(x>-1/2 & x<0) + 
  (x>0 & x<.5)*(-5*(x-0.2)^2 +1.075) + (x>.5)*(x+0.125); mu1 <- mu0; tau <- 0
a <- rbinom(n,1,ps); y <- a*mu1 + (1-a)*mu0 + rnorm(n,sd=(.2-.1*cos(2*pi*x)))

## estimate nuisance functions
pihat <- expit( logit(ps) + rnorm(n,mean=1/(n/4)^rate,sd=1/(n/4)^rate))
mu1hat <- predict(smooth.spline(x[a==1 & s==2],y[a==1 & s==2]),x)$y
mu0hat <- predict(smooth.spline(x[a==0 & s==2],y[a==0 & s==2]),x)$y

## construct estimators
plugin <- mu1hat-mu0hat
x1 <- predict(smooth.spline(x[a==1 & s==3],(y-mu0hat)[a==1 & s==3]),x)$y
x0 <- predict(smooth.spline(x[a==0 & s==3],(mu1hat-y)[a==0 & s==3]),x)$y
xl <- pihat*x0 + (1-pihat)*x1

pseudo <- ((a-pihat)/(pihat*(1-pihat)))*(y-a*mu1hat-(1-a)*mu0hat) + mu1hat-mu0hat
drl <- predict(smooth.spline(x[s==3],pseudo[s==3]),x)$y
pseudo.or <- ((a-ps)/(ps*(1-ps)))*(y-a*mu1-(1-a)*mu0) + mu1-mu0
oracle.drl <- predict(smooth.spline(x[s==3],pseudo.or[s==3]),x)$y

## save MSEs
res$plugin[i] <- (n/4)*mean((plugin-tau)[s==4]^2)
res$xl[i] <- (n/4)*mean((xl-tau)[s==4]^2)
res$drl[i] <- (n/4)*mean((drl-tau)[s==4]^2)
res$oracle.drl[i] <- (n/4)*mean((oracle.drl-tau)[s==4]^2)

}

res2 <- rbind(res2, c(rate, apply(res,2,mean)))
}
\end{verbatim}

\bigskip

\noindent High-dimensional model example: 

\begin{verbatim}
set.seed(1234)
library(glmnet)
expit <- function(x){ exp(x)/(1+exp(x)) }; logit <- function(x){ log(x/(1-x)) }

nsim <- 100; res <- data.frame(matrix(nrow=nsim,ncol=4))
colnames(res) <- c("plugin","xl","drl","oracle.drl")

n <- 4*2000; d <- 500; alpha <- d/10; beta <- alpha

for (i in 1:nsim){

## simulate data
s <- sort(rep(1:4,n/4)); x <- matrix(rnorm(n*d), n, d)
mu0 <- expit(as.numeric(x %*% rep(c(1, 0), c(beta,d-beta)))/sqrt(beta/1))
ps <- expit(as.numeric(x %*% rep(c(1, 0), c(alpha,d-alpha)))/sqrt(alpha/0.25))
tau <- 0; mu1 <- mu0 + tau
a <- rbinom(n,1,ps); y <- rbinom(n,1,a*mu1+(1-a)*mu0)

## estimate nuisance functions
pihat <- predict(cv.glmnet(x[s==1,],a[s==1], family="binomial"),newx=x, 
         type="response", s="lambda.min")
mu0hat <- predict(cv.glmnet(x[a==0 & s==2,],y[a==0 & s==2], family="binomial"), newx=x, 
          type="response", s="lambda.min")
mu1hat <- predict(cv.glmnet(x[a==1 & s==2,],y[a==1 & s==2], family="binomial"),newx=x, 
          type="response", s="lambda.min")

## construct estimators
plugin <- mu1hat-mu0hat
x1 <- predict(cv.glmnet(x[a==1 & s==3,],(y-mu0hat)[a==1 & s==3]),newx=x)
x0 <- predict(cv.glmnet(x[a==0 & s==3,],(mu1hat-y)[a==0 & s==3]),newx=x)
xl <- pihat*x0 + (1-pihat)*x1

pseudo <- ((a-pihat)/(pihat*(1-pihat)))*(y-a*mu1hat-(1-a)*mu0hat) + mu1hat-mu0hat
drl <- predict(cv.glmnet(x[s==3,],pseudo[s==3]),newx=x)
pseudo.or <- ((a-ps)/(ps*(1-ps)))*(y-a*mu1-(1-a)*mu0) + mu1-mu0
oracle.drl <- predict(cv.glmnet(x[s==3,],pseudo.or[s==3]),newx=x)

## save MSEs
res$plugin[i] <- (n/4)*mean((plugin-tau)[s==4]^2)
res$xl[i] <- (n/4)*mean((xl-tau)[s==4]^2)
res$drl[i] <- (n/4)*mean((drl-tau)[s==4]^2)
res$oracle.drl[i] <- (n/4)*mean((oracle.drl-tau)[s==4]^2)

}
\end{verbatim}

%%%%%%%%%%%%
%%%%%%%%%%%%
%%%%%%%%%%%%
%%%%%%%%%%%%
%%%%%%%%%%%%

\end{document}